\documentclass[a4paper,10pt, english]{article}
\usepackage[latin1]{inputenc}
\usepackage{babel}
\usepackage{amsmath,amsfonts,amssymb,amsthm}
\usepackage{bbm}
\usepackage{hyperref}
\usepackage{verbatim}
\usepackage{tikz}
\usetikzlibrary{patterns}
\usepackage{pstricks}

\theoremstyle{plain}
\newtheorem{teo}{Theorem}[section]

\newtheorem{prop}[teo]{Proposition}
\newtheorem{lem}[teo]{Lemma}
\newtheorem{claim}[teo]{Claim}
\newtheorem{rem}[teo]{Remark}
\newtheorem{defin}[teo]{Definition}
\newtheorem{subclaim}[teo]{Subclaim}
\newtheorem{open}[teo]{Open Problem}

\newcommand{\system}[1]{\mbox{\fontfamily{cmss}\fontshape{n}\fontseries{m}\selectfont#1}}
\newcommand{\ZF}{\system{ZF}}
\newcommand{\ZFC}{\system{ZFC}}
\newcommand{\AC}{\system{AC}}
\newcommand{\GCH}{\system{GCH}}

\newcommand{\DC}{\system{DC}}
\newcommand{\AD}{\system{AD}}

\DeclareMathOperator{\cof}{cof}
\DeclareMathOperator{\crt}{crt}
\DeclareMathOperator{\OD}{OD}
\DeclareMathOperator{\HOD}{HOD}
\DeclareMathOperator{\Ord}{Ord}
\DeclareMathOperator{\Ult}{Ult}

\DeclareMathOperator{\dom}{dom}

\DeclareMathOperator{\Fml}{Fml}

\title{Generic I0 at $\aleph_\omega$}
\author{Vincenzo Dimonte\footnote{\emph{E-mail address:} \texttt{vincenzo.dimonte@gmail.com}\newline 
\textsc{Kurt G\"{o}del Research Center}, University of Vienna, W\"{a}hringer Stra\ss e 25, A-1090 Vienna, Austria\newline 
\emph{Current address}: \textsc{Universit\`{a} degli Studi di Udine}, Dipartimento di Scienze Matematiche, Informatiche e Fisiche, via delle Scienze, 206, 33100 Udine, Italy}}

\begin{document}

\maketitle

\begin{abstract}

In this paper it is introduced a generic large cardinal akin to I0, and the consequences of $\aleph_\omega$ being such a generic large cardinal. In this case $\aleph_\omega$ is J\'{o}nsson, and in a choiceless inner model many properties hold that are in contrast with PCF in ZFC.



\end{abstract}

 \section{Introduction}

  What are generic large cardinals? As there is no specific definition of the notion of large cardinal, we can expect the same for generic large cardinals. Yet one can identify some vague pattern: All the large cardinal axioms of a certain strength postulate the existence of an elementary embedding of the form $j:V\prec M$, where $M$ is an inner model of $V$ with some closure properties, and the relevant large cardinal is the critical point of the embedding (or, more rarely, the supremum of critical points of such embeddings); therefore we call generic large cardinals the critical points of elementary embeddings of the form $j:V\prec M\subseteq V[G]$, where $M$ is an inner model of some generic extension of $V$. We can date back the origin of generic large cardinals to the introduction by Solovay of the generic ultrapower, and then to the definition by Jech and Prikry in \cite{JechPrikry} of the precipitous ideal, whose generic ultrapower gives exactly the situation described above. Of course, the same blueprint of a generic large cardinal can generate axioms of wildly different consistency strength, depending how ``close'' is $V[G]$ from $V$ (for example whether class forcing is permitted or not), how large (or small) we want it to be, and even considering variations that stray from the pattern but are conceptually close.
	
	The advantage of generic large cardinals embeddings is that their critical point can be small (for example $\omega_1$), therefore they are a tool to transfer properties of large cardinal to ``accessible'' cardinals. This works best for model-theoretical combinatorial properties, as the elementary embedding yields some reflection structure. For example:
	
	\begin{defin}
	 Given $\gamma_0,\gamma_1,\eta_0,\eta_1$ cardinals, we write $(\gamma_0,\gamma_1)\twoheadrightarrow (\eta_0,\eta_1)$ if the following holds: Given any language ${\cal L}$ that contains, possibly among other, a distinguished relation symbol, every model ${\cal U}=(U,R,\dots)$ in ${\cal L}$ such that $|U|=\gamma_0$ and $|R|=\gamma_1$ has an elementary submodel ${\cal V}=(V,S,\dots)$ such that $|V|=\eta_0$ and $|S|=\eta_1$.
	\end{defin}
	
	For example, Chang's Conjecture is $(\aleph_2,\aleph_1)\twoheadrightarrow(\aleph_1,\aleph_0)$. Moreover, a possible variation is to consider in the definition two, three, etc\dots distinguished relation symbols. 
	
	Such arrow structure is common in very large cardinals. As a reference for all the large cardinal notions that appear in this paper, we suggest Chapter 24 in \cite{Kanamori}.
	
	\begin{defin}
	 A cardinal $\kappa$ is \emph{huge} iff there exists $j:V\prec M$ such that $\kappa$ is the critical point of $j$ and $M^{j(\kappa)}\subseteq M$. 
	 
	 A cardinal $\kappa$ is \emph{$n$-huge} iff there exists $j:V\prec M$ such that $\kappa$ is the critical point of $j$ and $M^{\kappa_n}\subseteq M$, where $\kappa_0=\kappa$ and $\kappa_{n+1}=j(\kappa_n)$. 
	\end{defin}
	
	If $\kappa$ is huge and $j$ witnesses it, then for any $\gamma<\kappa$, $(j(\kappa),\gamma)\twoheadrightarrow(\kappa,\gamma)$. Kunen exploited this to prove Chang's Conjecture via a collapse of a huge cardinal (into a generic huge cardinal). Laver remarked that the same collapse can be used to prove $(\aleph_{n+2},\aleph_{n+1})\twoheadrightarrow(\aleph_{n+1},\aleph_n)$ for any $n\in\omega$. A more sophisticated collapse of a 2-huge cardinal was used by Foreman in \cite{Foreman} to prove $(\aleph_{n+3},\aleph_{n+2},\aleph_{n+1})\twoheadrightarrow(\aleph_{n+2},\aleph_{n+1},\aleph_n)$. 
	
	The dream is to push this to its limit: by a theorem of Silver (\cite{Eskew}), if $(\dots,\aleph_3,\aleph_2,\aleph_1)\twoheadrightarrow(\dots,\aleph_2,\aleph_1,\aleph_0)$, then $\aleph_\omega$ is J\'{o}nsson. It is a long-standing open problem whether $\aleph_\omega$ can be J\'{o}nsson, so proving this would be highly desirable. Unfortunately, the ``limit'' of hugeness, i.e., the existence of a $j:V\prec M$ that witnesses $n$-hugeness for every $n$ simultaneously, is actually inconsistent (the same proof for the inconsistency of the Reinhardt cardinal holds). Interestingly, also its generic counterpart is inconsistent: if $j:V\prec M\subseteq V[G]$, $\aleph_\omega=\sup_{n\in\omega}\kappa_n$, where $\kappa_0=\crt(j)<\aleph_\omega$ and $\kappa_{n+1}=j(\kappa_n)$, and $M^{\aleph_{\omega+1}}\subseteq M$, then $\aleph_{\omega+1}$ is J\'{o}nsson (see Remark \ref{jonsson}), but Shelah, using PCF theory, proved the following
	
	\begin{teo}[Shelah, \cite{Shelah3}]
	\label{nonJonsson}
	  If $2^{\aleph_0}\leq\aleph_{\omega+1}$, then $\aleph_{\omega+1}$ cannot be J\'{o}nsson.
	\end{teo}
	
	Yet there are large cardinals between $n$-huge and Reinhardt:
	
	\begin{defin}
	 \begin{itemize}
	  \item I1: there exists $\lambda$ such that $\exists j:V_{\lambda+1}\prec V_{\lambda+1}$;
		\item I0: there exists $\lambda$ such that $\exists j:L(V_{\lambda+1})\prec L(V_{\lambda+1})$, with $\crt(j)<\lambda$.
	 \end{itemize}
	\end{defin}
	
	I0 in particular is very fruitful in consequences, as it gives to $L(V_{\lambda+1})$ a structure that is very similar to the structure of $L(\mathbb{R})$ under the Axiom of Determinacy, in a way that it is still not completely understood: if $\AD$ holds in $L(\mathbb{R})$, then it is well-known that $\omega_1$ is measurable in $L(\mathbb{R})$, but there are also many other measurable cardinals above and the first cardinal ``above'' $\mathbb{R}$ is a limit of them (\cite{Kechris}). The same holds in $L(V_{\lambda+1})$ under I0, where $\lambda^+$ is measurable and the first cardinal ``above'' $V_{\lambda+1}$ is a limit of measurable cardinals (see Lemma 22 in \cite{Woodin} for the original more general case, and Theorem 6.7 in \cite{Dimonte0} specifically for the I0 case).
	
	Now, the ``right'' versions of generic I1 and generic I0 (see Definition \ref{genericI0}) prove that $\aleph_\omega$ is J\'{o}nsson. Moreover, this paper presents other results, in line with the results about I0. In short, under generic I0 the structure of $L({\cal P}(\aleph_\omega))$ is similar to the structure of $L(\mathbb{R})$ under the \AD:
	
\newtheorem*{thm:main}{Theorem \ref{thm:main}}	
\begin{thm:main}
  Suppose that generic I0 holds at $\aleph_\omega$ as witnessed by a forcing notion $\mathbb{P}$. Then $\aleph_\omega$ is J\'{o}nsson. Moreover, in $L({\cal P}(\aleph_\omega))$ (see below for all the definitions):
	\begin{itemize}
	 \item if $\aleph_{\omega+1}^{V[G]}=\aleph_{\omega+1}$, then $\aleph_{\omega+1}$ is $\omega$-strongly measurable and J\'{o}nsson;
	 \item $\Theta$ is ``inaccessible'', in the sense that it is regular and if $\alpha<\Theta$ then there is a surjection from ${\cal P}(\aleph_\omega)$ to ${\cal P}(\alpha)$;
	 \item if $\mathbb{P}$ is $\omega$-closed, then $\Theta$ is limit of $\omega$-strongly measurable cardinals that are measurable.
	\end{itemize}
 \end{thm:main}
	
Some of these results are not new. A model where $\aleph_{\omega+1}$ is a measurable cardinal was already known:

 	 \begin{teo}[Martin]
	 \label{Martin}
   \AD{} implies $\aleph_{\omega+1}$ is measurable.
  \end{teo}
  
  \begin{teo}[Apter, \cite{Apter}]
	\label{Apter}
   Suppose $\kappa$ is $2^\lambda$-supercompact, with $\lambda$ measurable. Then there is a model of $\ZF + \aleph_{\omega+1}$ is measurable.
  \end{teo}
	
New is the J\'{o}nsson-ness of $\aleph_{\omega+1}$\footnote{It is well known in $\ZF+\DC${} that every measurable cardinal is Ramsey \cite{ErdosHajnal} and every Ramsey cardinal is J\'{o}nsson \cite{ErdosHajnal2}, so Theorems \ref{Martin} and \ref{Apter} seem to prove that $\aleph_{\omega+1}$ is J\'{o}nsson. But is this true also in \ZF{}?}, that goes strongly against Shelah's Theorem \ref{nonJonsson}, but this is not surprising, as under generic I0 $L({\cal P}(\aleph_\omega))$ does not satisfy the Axiom of Choice, and PCF theory has not been fully developed in this setting. But in fact, Theorem \ref{thm:main} goes very strongly against a more celebrated result of PCF theory:
\begin{teo}[Shelah, \cite{Shelah}]
 If $\aleph_\omega$ is strong limit, then $2^{\aleph_\omega}<\aleph_{\omega_4}$.
\end{teo}
	In Theorem \ref{thm:main}, $\Theta$ can be seen as a measure of $2^{\aleph_\omega}$ in a choiceless setting: it is drastically larger than the cap given by Shelah's theorem, as it is an inaccessible limit of measurable cardinals. Generic I0 therefore not only proves that $\aleph_\omega$ is J\'{o}nsson, but provides a model where PCF theory fails radically.
	
 In this paper no attempt is made to prove the consistency of generic I0. Currently it is not even known whether generic 3-huge is consistent, as Foreman's technique meets difficult technical obstacles in its generalization. This is the graph indicating current knowledge, the arrows are implication of consistency, framed are the new results, between parentheses the hypotheses that are inconsistent with \ZFC:

\begin{tikzpicture}
  \draw (1,7) node{(Reinhardt)};
  \draw (1,1) node{huge};
	\draw (1,2) node{2-huge};
	\draw (1,3) node{3-huge};
	\draw (1,4) node{\vdots};
	\draw [->] (1,1.75) -- (1,1.25);
	\draw [->] (1,2.75) -- (1,2.25);
	\draw [->] (1,3.70) -- (1,3.25);
	\draw (3.7,1) node{generic huge};
	\draw (9,1) node{$(\aleph_{n+2},\aleph_{n+1})\twoheadrightarrow(\aleph_{n+1},\aleph_n)$};
	\draw [->] (1.5,1) -- (2.5,1);
	\draw [->] (4.8,1) -- (6.7,1);
	\draw (3.7,2) node{generic 2-huge};
	\draw (9,2) node{$(\aleph_{n+3},\aleph_{n+2},\aleph_{n+1})\twoheadrightarrow(\aleph_{n+2},\aleph_{n+1},\aleph_n)$};
	\draw [->] (5,2) -- (6,2);
	\draw (3.7,3) node{generic 3-huge};
	\draw (9,3) node{$(\aleph_{n+4},\aleph_{n+3},\aleph_{n+2},\aleph_{n+1})\twoheadrightarrow(\aleph_{n+3},\aleph_{n+2},\aleph_{n+1},\aleph_n)$};
	\draw [->] (4.9,3) -- (5.1,3);
	\draw [->] (3.7,1.75) -- (3.7,1.25);
	\draw [->] (3.7,2.75) -- (3.7,2.25);
	\draw [->] (3.7,3.7) -- (3.7,3.25);
	\draw [->] (9,1.75) -- (9,1.25);
	\draw [->] (9,2.75) -- (9,2.25);
	\draw [->] (9,3.7) -- (9,3.25);
	\draw (3.7,4) node{\vdots};
	\draw (9,4) node{\vdots};
	\draw (9,5) node{$\aleph_\omega$ J\'{o}nsson};
	\draw (3.7,7) node{(generic Reinhardt)};
	\draw (1,5) node{I1};
	\draw (1,6) node{I0};
	\draw [->] (1,5.75) -- (1,5.25);
	\draw [->] (1,4.75) -- (1,4.25);
	\draw (3.7,5) node{generic I1};
	\draw (3.7,6) node{generic I0};
	\draw [->] (5,5) -- (7.5,5);
	\draw [->] (1.75,2) -- (2.5,2);
		\draw (9,6) node{$\aleph_\omega$ J\'{o}nsson, Theorem \ref{thm:main}};
	\draw [->] (5,6) -- (6.75,6);
		\draw [->] (3.7,5.75) -- (3.7,5.25);
	\draw [->] (1,6.75) -- (1,6.25);
	\draw [->] (3.7,6.75) -- (3.7,6.25);
	\draw [->] (3.7,4.75) -- (3.7,4.25);
	\draw [->] (9,4.75) -- (9,4.25);
	\draw [->] (9,5.75) -- (9,5.25);
	\draw (2.5,6.5) -- (12,6.5)-- (12,4.5) -- (2.5,4.5) -- (2.5,6.5);
 \end{tikzpicture}

 \section{Preliminaries}
  
	This paper is about $L({\cal P}(\aleph_\omega))$. The assumption is that there is a $V$ in which \ZFC{} holds, and $L({\cal P}(\aleph_\omega))$ is constructed inside $V$, as the smallest model of \ZF{} that contains the ``real'' ${\cal P}(\aleph_\omega)$. The construction is the same as $L$, but the starting point of the construction is $L_0({\cal P}(\aleph_\omega))={\cal P}(\aleph_\omega)$, and the levels $L_\alpha({\cal P}(\aleph_\omega))$ are defined as expected.
	
	Even if \AC{} holds in $V$, it does not necessarily hold in $L({\cal P}(\aleph_\omega))$. Yet some relatively strong weak form of choice is inherited:
	
	\begin{defin}
	 \emph{$DC_\lambda$}: $\forall X\forall F:(X)^{<\lambda}\to{\cal P}(X)\setminus\{\emptyset\}\quad\exists g:\lambda\to X\quad\forall\gamma<\lambda\  g(\gamma)\in F(g\upharpoonright\gamma)$.
	\end{defin}
	
  \begin{rem}[$\ZF+\DC_{\aleph_\omega}$]
  \label{DC}
   $L({\cal P}(\aleph_\omega))\vDash \DC_{\aleph_\omega}$
  \end{rem}
  \begin{proof}[Sketch of proof]
   The proof exploits the fact that there exists a surjection $\Phi:\Ord\times{\cal P}(\aleph_\omega)\twoheadrightarrow L({\cal P}(\aleph_\omega))$: $\DC_{\aleph_\omega}$ is true on ${\cal P}(\aleph_\omega)$ as it is true in $V$, on $\Ord$ because $\Ord$ is well-ordered, and $\Phi$ mixes the two to carry $\DC_{\aleph_\omega}$ in all $L({\cal P}(\aleph_\omega))$. For a similar proof in $L(V_{\lambda+1})$ see Lemma 5.10 in \cite{Dimonte0}.
  \end{proof}
	
	In a choiceless model the concept of cardinality is notoriously more complex. If $L({\cal P}(\aleph_\omega))\nvDash\AC$, then ${\cal P}(\aleph_\omega)$ must not be well-orderable, otherwise the well-order extends to all $L({\cal P}(\aleph_\omega))$. How to characterize the cardinality of ${\cal P}(\aleph_\omega)$, then? In $L(\mathbb{R})$ the classical way is to calculate the largeness of $\mathbb{R}$ via surjections instead of the usual bijections, and we borrow the same idea. Surjections in this paper are always noted with the symbol $\twoheadrightarrow$.
	
	\begin{defin}
    $\Theta^{L({\cal P}(\aleph_\omega))}=\{\alpha:\exists\pi\in L({\cal P}(\aleph_\omega)),\ \pi:{\cal P}(\aleph_\omega)\twoheadrightarrow\alpha\}$.
  \end{defin}
  
	For the rest of the paper, $\Theta^{L({\cal P}(\aleph_\omega))}=\Theta$, as this will not create confusion. It is routine to prove that $\Theta$ is regular in $L({\cal P}(\aleph_\omega))$, see for example Exercise 28.19 in \cite{Kanamori}.
	
	The cardinal $\Theta$ is central in the study of the subsets of ${\cal P}(\aleph_\omega)$:
	\begin{lem}[Folklore]
	\label{folklore}
	  If $X\in L({\cal P}(\aleph_\omega))$, $X\subseteq {\cal P}(\aleph_\omega)$, then $X\in L_\Theta({\cal P}(\aleph_\omega))$.
	\end{lem}
	
	For this reason models like $L_\alpha({\cal P}(\aleph_\omega))$ with $\alpha<\Theta$ are also relevant in the study of the subsets of ${\cal P}(\aleph_\omega)$, as they form a well-ordered set of approximations of $L_\Theta({\cal P}(\aleph_\omega))$. We say that an ordinal $\alpha<\Theta$ is \emph{good} iff  every element of $L_\alpha({\cal P}(\aleph_\omega))$ is definable in $L_\alpha({\cal P}(\aleph_\omega))$ from an element in ${\cal P}(\aleph_\omega)$. Good ordinals in $L(V_{\lambda+1})$ were introduced by Laver in \cite{Laver}, and they are described in more details in \cite{DimonteFriedman}.
	
	\begin{lem}
  \label{ubdgood}
    There are unbounded good ordinals below $\Theta$.
  \end{lem}
  \begin{proof}[Sketch of proof]
    As in the $L(V_{\lambda+1})$ case, see \cite{Laver}. The key point is that ${\cal P}(\aleph_\omega)$ can be considered closed by $\omega$-sequences\footnote{Actually, ${\cal P}(\aleph_\omega)$ can be considered closed under $\aleph_\omega$-sequences. Given a pairing function $\langle\cdot,\cdot\rangle:\aleph_\omega\times\aleph_\omega\to\aleph_\omega$, we code $\langle A_\alpha:\alpha<\aleph_\omega\rangle$ as $A^*=\{\langle\alpha,a\rangle:a\in A_\alpha\}$.}.
    
    Let $\beta_0<\Theta$. Then there is a surjection $h_0:{\cal P}(\aleph_\omega)\twoheadrightarrow L_{\beta_0}({\cal P}(\aleph_\omega))$ in $L({\cal P}(\aleph_\omega))$. Therefore there must exist $a_0\in{\cal P}(\aleph_\omega)$ such that $h_0$ is definable from $a_0$, and since $h_0$ is codeable as a subset of ${\cal P}(\aleph_\omega)$ by Lemma \ref{folklore} there exists $\beta_1<\Theta$ such that $h_0\in L_{\beta_1}({\cal P}(\aleph_\omega))$. By induction define $h_n$, $\beta_n$ and $a_n$ in the obvious way, and it is routine to check that if $\beta_\omega$ is the supremum of the $\beta_n$'s, then there exists $h:{\cal P}(\aleph_\omega)\twoheadrightarrow L_{\beta_\omega}({\cal P}(\aleph_\omega))$ that is definable from $\langle a_0,\dots,a_n,\dots\rangle\in{\cal P}(\aleph_\omega)$, and therefore $\beta_\omega$ is good. 
  \end{proof}
		
	Large cardinals are also problematic in choiceless models. Fortunately, one definition of ``measurable'' can be adopted unchanged in \ZF: we say that $\kappa$ is a measurable cardinal if there exists a $\kappa$-complete non-principal ultrafilter on $\kappa$. But the study of $L(\mathbb{R})$ under the Axiom of Determinacy indicates a more specific notion of measurability, where the measure comes from the club filter.
	
	  \begin{defin}
   $E^\kappa_\omega=\{\eta<\kappa:\cof(\eta)=\omega\}$.
  \end{defin}
	
	\begin{defin}[Woodin, Definition 198 in \cite{Woodin2}]
   Let $\kappa$ be an uncountable regular cardinal. We say that $\kappa$ is $\omega$-strongly measurable if there exists $\gamma<\kappa$ such that $2^\gamma<\kappa$ and there does not exist a sequence $\langle S_\alpha:\alpha<\gamma\rangle$ of pairwise disjoint subsets of $\kappa$ such that for each $\alpha<\gamma$, $S_\alpha$ is stationary in $E^\kappa_\omega$.
  \end{defin}
	
  Under the right conditions, an $\omega$-strongly measurable is actually measurable:

 \begin{rem}[$\ZF+\DC_\gamma$]
   \label{stronglyornot}
   If $\kappa$ is $\omega$-strongly measurable as witnessed by $\gamma$ and the club filter on $\kappa$ is $\kappa$-complete, then $\kappa$ is measurable.
  \end{rem}
  \begin{proof}[Sketch of proof]
   Let $\gamma<\kappa$ that witnesses that $\kappa$ is strongly measurable. It is a standard procedure (see Lemma 2.6 in \cite{Kafkoulis}) to prove that if there does not exist a sequence $\langle S_\alpha:\alpha<\gamma\rangle$ of pairwise disjoint subsets of $\kappa$ such that for each $\alpha<\gamma$, $S_\alpha$ is stationary in $E^\kappa_\omega$, then there exists a stationary set on which the club filter is an ultrafilter: 
   
   If a stationary set cannot be split into two stationary sets (modulo a nonstationary set), then the club filter is an ultrafilter on it. So suppose this never happens. Then by induction one can split $\kappa$ in two stationary sets, then in four, and so on. For the limit case, one uses $\DC_\gamma$, and after $\gamma$ steps a partition $\langle S_\alpha:\alpha<\gamma\rangle$ of stationary sets is created, and that is contradictory.
   
   The club filter is $\kappa$-complete, so the remark follows.
  \end{proof}
	
	Note that under $\ZF+\DC_{<\kappa}$ the club filter on $\kappa$ is $\kappa$-complete, therefore in that case it is a direct implication.

 Another large cardinal that is expressible in a choiceless context is the J\'{o}nsson cardinal:

 \begin{defin}
  A cardinal $\kappa$ is J\'{o}nsson iff any structure for a countable first-order language with domain of cardinality $\kappa$ has a proper elementary substructure with domain of the same cardinality.
 \end{defin}

 This concept is tightly connected to generic large cardinals:

 \begin{rem}
  \label{jonsson}
  Let $\lambda$ be a cardinal, $G$ be generic for some forcing notion, $j:V\prec M\subseteq V[G]$, $j(\lambda)=\lambda$ and $^\lambda M\subseteq M$, or just $j''\lambda\in M$. Then $\lambda$ is J\'{o}nsson.
 \end{rem}
 \begin{proof}
  Suppose not. Then there exists some structure ${\cal U}=(U,\in,\dots)$ such that $|U|=\lambda$ and ${\cal U}$ has no elementary substructure with domain $U$ of cardinality $\lambda$. We can suppose that $U\subseteq\lambda$. By elementarity, since $\lambda$ is a fixed point of $j$, in $M$ the same is true for $j({\cal U})=(j(U),\in,\dots)$. But consider $j''{\cal U}$. We have that $j''{\cal U}\in M$, either because $^\lambda M\subseteq M$ or because $j''{\cal U}=j({\cal U})\cap j''\lambda\in M$, but $j''{\cal U}\prec j({\cal U})$, this is a contradiction.
 \end{proof}
  
 \section{Definition of generic I0}
  
	Recall that I1 is the existence of an elementary embedding $j:V_{\lambda+1}\prec V_{\lambda+1}$ and that I0 is the existence of an elementary embedding $j:L(V_{\lambda+1})\prec L(V_{\lambda+1})$ with critical point less than $\lambda$. By the proof of Reinhardt's inconsistency, in these cases $\lambda$ must be a strong limit cardinal of cofinality $\omega$, and $j(\lambda=\lambda$. Moreover, under I0 it must be that $L(V_{\lambda+1})\nvDash\AC$. 
	
	As it is usual for generic large cardinals, the concept of ``generic I0'' is very vague. It implies that the generic large cardinal should be similar to I0, i.e., the existence of some embedding $j:L({\cal P}(\kappa))\prec L({\cal P}(\kappa))^{V[G]}$, but this is not enough to characterize it, as different $\kappa$'s and variations on how similar is $V[G]$ to $V$ produce drastically different hypotheses. The first step will be to restrict ourselves to the case $\kappa=\aleph_\omega$, we are going to call it ``generic I0 holds at $\aleph_\omega$''. Then we indicate other properties that state that $V[G]$ is close enough to $V$ (the closer it is, the harder it is for it to be consistent, but on the other hand there are more interesting consequences). This choice has of course some arbitrariness: it is possible for different kind of closures to be fruitful, so the definition in all its details should be considered limited to this paper. But every point of the following definition has a reason to be there for the purpose of the paper, and the rest of the section is dedicated to the explanation of such reasons.
	
  \begin{defin}
  \label{genericI0}
   Suppose \GCH{} below $\aleph_\omega$. We say that generic I0 holds at $\aleph_\omega$ if there exists a forcing notion $\mathbb{P}\in L({\cal P}(\aleph_\omega))$ and a generic $G$ $\mathbb{P}$-generic such that:
     \begin{enumerate}
		\item in $L({\cal P}(\aleph_\omega))$ there exists $\pi:{\cal P}(\aleph_\omega)\twoheadrightarrow\mathbb{P}$;
		\item $\aleph_\omega=\aleph_\omega^{V[G]}$;
   \item every element of ${\cal P}(\aleph_\omega)^{V[G]}$ has a name (coded) in ${\cal P}(\aleph_\omega)$;
   \item in $V[G]$ there exists $j:L({\cal P}(\aleph_\omega))\prec L({\cal P}(\aleph_\omega))[G]$ with $\crt(j)<\aleph_\omega$;
   \item there is a $\mathbb{P}$-term for $H^{V[G]}(\aleph_\omega)$ and $j\upharpoonright H(\aleph_\omega):H(\aleph_\omega)\prec H^{V[G]}(\aleph_\omega)$.
  \end{enumerate}
  \end{defin}
	
	Point (1) implies that we can suppose that $\mathbb{P}\subseteq {\cal P}(\aleph_\omega)$, and we will for the rest of the paper.
	
	Point (4) needs to be discussed in more detail. We have two possible choices for the codomain of $j$:
	\begin{itemize}
	 \item $L({\cal P}(\aleph_\omega))^{V[G]}$, that is $L({\cal P}(\aleph_\omega))$ as built inside the forcing extension;
	 \item $L({\cal P}(\aleph_\omega))[G]$, that is the forcing extension of $L({\cal P}(\aleph_\omega))$, or the class of the elements of $V[G]$ that have a name in $L(V_{\lambda+1})$.
	\end{itemize}
	
	Definition \ref{genericI0}(2) and (3) provide different ways to describe these two models. First of all, by Definition \ref{genericI0}(2) 
	\begin{equation*}
		{\cal P}(\aleph_\omega)^{V[G]}={\cal P}^{V[G]}(\aleph_\omega^{V[G]})={\cal P}^{V[G]}(\aleph_\omega).\footnote{With ${\cal P}^{V[G]}(X)$ we indicate of course the subsets of $X$ that are in $V[G]$}
	\end{equation*}
	Then Definition \ref{genericI0}(3) says that ${\cal P}(\aleph_\omega)^{V[G]}\subseteq{\cal P}(\aleph_\omega)[G]$, therefore ${\cal P}^{V[G]}(\aleph_\omega)={\cal P}^{{\cal P}(\aleph_\omega)[G]}(\aleph_\omega)$. In other words, all the subsets of $\aleph_\omega$ in $V[G]$ are already in ${\cal P}(\aleph_\omega)[G]$, so for example ${\cal P}^{L({\cal P}(\aleph_\omega))[G]}(\aleph_\omega)={\cal P}^{V[G]}(\aleph_\omega)$. But then $L({\cal P}(\aleph_\omega))[G]$ is a \ZF-model that contains ${\cal P}^{V[G]}(\aleph_\omega)$, therefore trivially 
	\begin{equation*}
	 (L({\cal P}(\aleph_\omega)))^{V[G]}=L({\cal P}^{V[G]}(\aleph_\omega))\subseteq L({\cal P}(\aleph_\omega))[G].
	\end{equation*}
	
	Could these two models be equal? The key object is ${\cal P}(\aleph_\omega)$: it is trivially in $L({\cal P}(\aleph_\omega))[G]$, but not necessarily in $L({\cal P}(\aleph_\omega))^{V[G]}$. But this is a consequence of generic I0 at $\aleph_\omega$:
  
	\begin{lem}
	\label{equal}
   Suppose that generic I0 holds at $\aleph_\omega$ and let $\mathbb{P}$, $G$ witness that. Then $(L({\cal P}(\aleph_\omega)))^{V[G]}=L({\cal P}(\aleph_\omega))[G]$. 
  \end{lem}
	\begin{proof}
	 Naturally $L({\cal P}(\aleph_\omega))\vDash V=L({\cal P}(\aleph_\omega))$, therefore by elementarity $L({\cal P}(\aleph_\omega))[G]\vDash V=L({\cal P}(j(\aleph_\omega)))$. By Definition \ref{genericI0}(2) $j(\aleph_\omega)=\aleph_\omega$, so $L({\cal P}(\aleph_\omega))[G]\vDash V=L({\cal P}(\aleph_\omega))$. By Definition \ref{genericI0}(3) ${\cal P}^{L({\cal P}(\aleph_\omega))[G]}(\aleph_\omega)={\cal P}^{V[G]}(\aleph_\omega)$, therefore
	\begin{equation*}
	L({\cal P}(\aleph_\omega))[G]=L({\cal P}^{V[G]}(\aleph_\omega))=(L({\cal P}(\aleph_\omega)))^{V[G]}
	\end{equation*}
	\end{proof}
	
	The proof also provides that $L({\cal P}(\aleph_\omega))^{V[G]}=L({\cal P}(\aleph_\omega))^{L({\cal P}(\aleph_\omega))[G]}$, and for better readability we will use the first notation even when we mean the second one.
	
	We can try to be even more precise: any element of $L({\cal P}(\aleph_\omega))^{V[G]}$ has a name in $L({\cal P}(\aleph_\omega))$, but we can do this level by level. It is a calculation that if $\mathbb{P}\in L_\gamma({\cal P}(\aleph_\omega))$ and $\sigma\in L_\eta({\cal P}(\aleph_\omega))$ is a $\mathbb{P}$-name, then $Def(\sigma_G)$ has a name in $L_{\max\{\gamma,\eta\}+3}({\cal P}(\aleph_\omega))$\footnote{For example, $\{\langle\tau,1\rangle:\dom(\tau)\subseteq\dom(\sigma),\ \exists\varphi\in\Fml\ \exists\tau_1,\dots\tau_n\in\dom(\sigma)\ \forall p\in\mathbb{P}\ \forall\pi\in\dom(\sigma)\ \langle\pi,p\rangle\in\tau\leftrightarrow p\Vdash\sigma\vDash\varphi(\pi,\sigma,\tau_1,\dots,\tau_n)\}$.}. Now, ${\cal P}^{V[G]}(\aleph_\omega)$ has a name in $L_\omega({\cal P}(\aleph_\omega))$, so by induction $L_\delta({\cal P}(\aleph_\omega))^{V[G]}$ has a name in $L_{\gamma+\delta+n}({\cal P}(\aleph_\omega))$ for some $n\in\omega$. With a similar calculation, we find that if $x$ is definable over a model $\sigma_G$, with $\sigma\in L_\eta({\cal P}(\aleph_\omega))$, then $x$ has a name in $L_{\max\{\gamma,\eta\}+3}({\cal P}(\aleph_\omega))$\footnote{For example $\{\langle\pi,p\rangle\in\dom(\sigma)\times\mathbb{P}: p\Vdash\pi\in\sigma\wedge\sigma\vDash\varphi(\pi,\tau_1,\dots,\tau_n)\}$}, so if $x\in L_{\delta+1}({\cal P}(\aleph_\omega))^{V[G]}$, then $x$ has a name in $L_{\gamma+\delta+n}({\cal P}(\aleph_\omega))$ for some $n\in\omega$. In other words, for any $\delta$ ordinal, there is a $n\in\omega$ such that $L_{\delta+1}({\cal P}(\aleph_\omega))^{V[G]}\subseteq L_{\gamma+\delta+n}({\cal P}(\aleph_\omega))$.
	
	On the other hand, suppose that $\eta$ is such that ${\cal P}(\aleph_\omega),\ G\in L_\eta({\cal P}(\aleph_\omega))^{V[G]}$. Then any element in $L_\delta({\cal P}(\aleph_\omega))$, $\mathbb{P}$-names included, will be in $L_{\max\{\delta,\eta\}}({\cal P}(\aleph_\omega))^{V[G]}$, therefore for any $\delta$ ordinal $L_\delta({\cal P}(\aleph_\omega))[G]\subseteq L_{\max\{\delta,\eta\}+1}({\cal P}(\aleph_\omega))^{V[G]}$. We proved the following:
	
	\begin{lem}
	\label{level}
	 Suppose that generic I0 holds at $\aleph_\omega$ and let $\mathbb{P}$, $G$ witness that. Let $\gamma<\Theta$ be such that $\mathbb{P},\ {\cal P}(\aleph_\omega),\ G\in L_\gamma({\cal P}(\aleph_\omega))^{V[G]}$. Then for any $\gamma<\alpha$ limit ordinal, $L_\alpha({\cal P}(\aleph_\omega))^{V[G]}=L_\alpha({\cal P}(\aleph_\omega))[G]$.
	\end{lem}
	
	Lemma \ref{equal} shows that the definition of generic I0 with $L({\cal P}(\aleph_\omega))[G]$ implies the one with $(L({\cal P}(\aleph_\omega)))^{V[G]}$. It is not much stronger, for example assuming $\omega$-closure is enough to make them equivalent:
	
  \begin{lem}
	 \label{omegaclosed}
   Let $\mathbb{P}$ be an $\omega$-closed forcing notion on $L({\cal P}(\aleph_\omega))$ that with $G$ $\mathbb{P}$-generic satisfies Definition \ref{genericI0}(1), (2), (3). Then $(L({\cal P}(\aleph_\omega)))^{V[G]}=L({\cal P}(\aleph_\omega))[G]$.
  \end{lem}
  \begin{proof}
   For the case $L(V_{\lambda+1})$, this is Corollary 3.8 in \cite{DimonteFriedman}. As noted, we just need to prove that ${\cal P}(\aleph_\omega)\in(L({\cal P}(\aleph_\omega)))^{V[G]}$. Let $G$ be generic for $\mathbb{P}$. By Definition \ref{genericI0}(1), we can suppose $G\subseteq{\cal P}(\aleph_\omega)$. As $\mathbb{P}$ is $\omega$-closed, we can define 
	\begin{equation*}
	 {\cal P}(\aleph_\omega)=\{a\in{\cal P}^{V[G]}(\aleph_\omega):\forall n\in\omega\ a\cap\aleph_n\in V\}.
	\end{equation*}
	
	This is because if $a\cap\aleph_n=a_n\in V$, then $\langle a_n:n\in\omega\rangle\in V$, and $a=\bigcup_{n\in\omega}a_n\in V$.  So $L({\cal P}(\aleph_\omega))[G]\subseteq L({\cal P}(\aleph_\omega)))^{V[G]}$.
  \end{proof}
	
	Theorem 8.8 in \cite{Dimonte0} proves that in the $L(V_{\lambda+1})$ case the $\omega$-closure is necessary and sufficient.
	
	\begin{open}
	 \label{perfect}
	 Let $\mathbb{P}$ be a forcing notion on $L({\cal P}(\aleph_\omega))$ that satisfies Definition \ref{genericI0}(1), (2) and (3),  and such that $(L({\cal P}(\aleph_\omega)))^{V[G]}=L({\cal P}(\aleph_\omega))[G]$. Must $\mathbb{P}$ be $\omega$-closed?
	\end{open}
	
	There are therefore two approaches for defining generic I0. One is to ask for $\mathbb{P}$ to be $\omega$-closed: this will give the full results of Theorem \ref{thm:main}, and because of Lemma \ref{omegaclosed} the codomain of the generic embedding can be just $L({\cal P}(\aleph_\omega))^{V[G]}$. The other is to forgo $\omega$-closure, but then adding the (potentially) stronger condition of the codomain of the embedding to be $L({\cal P}(\aleph_\omega))[G]$. This will prove the first two points of Theorem \ref{thm:main}. Why are we using this stronger version of generic I0, even if it would seem that the weaker version is closer to the spirit of generic large cardinals? It is just a technical matter: we are proving now that, thanks to the lifting lemma, $L({\cal P}(\aleph_\omega))[G]$ maintains the ``reflection structure'' of $L({\cal P}(\aleph_\omega))$, and this is central to the proofs. For this paper, we are not choosing one of the two approaches, especially in light of the fact that the Open Problem \ref{perfect} would prove that the two approaches are the same.
	
	\begin{lem}
   \label{lifting}
	 Suppose generic I0 at $\aleph_\omega$, and let $\mathbb{P}, G$ witness it. Let $\gamma$ be such that $\mathbb{P},\ G,\ {\cal P}^V(\aleph_\omega)\in L_\gamma({\cal P}(\aleph_\omega))^{V[G]}$. Then for any $\gamma<\alpha,\beta$ limit ordinals, $L_\alpha({\cal P}(\aleph_\omega))\prec L_\beta({\cal P}(\aleph_\omega))$ iff
	 \begin{equation*}
		   L({\cal P}(\aleph_\omega))[G]\vDash L_\alpha({\cal P}(\aleph_\omega))\prec L_\beta({\cal P}(\aleph_\omega)).
		 \end{equation*}
\end{lem}
\begin{proof}
 From left to right: By the usual lifting lemma\footnote{It says that if $j:M\prec N$, $\mathbb{P}\in M$, $G$ $\mathbb{P}$-generic, $H$ $j(\mathbb{P})$-generic and $j''G\subseteq H$, then $j$ extends to $M[G]\prec N[H]$}, as $\mathbb{P}\in L_\alpha({\cal P}(\aleph_\omega))$, $L_\alpha({\cal P}(\aleph_\omega))[G]\prec L_\beta({\cal P}(\aleph_\omega))[G]$. But by Lemma \ref{level} $L_\alpha({\cal P}(\aleph_\omega))[G]=L_\alpha({\cal P}(\aleph_\omega))^{V[G]}$ and $L_\beta({\cal P}(\aleph_\omega))[G]=L_\beta({\cal P}(\aleph_\omega))^{V[G]}$, so 
\begin{equation*}
		   V[G]\vDash L_\alpha({\cal P}(\aleph_\omega))\prec L_\beta({\cal P}(\aleph_\omega)).
\end{equation*}
but then by absoluteness one direction is proved.

On the other hand, suppose that
     \begin{equation*}
		   L({\cal P}(\aleph_\omega))[G]\vDash L_\alpha({\cal P}(\aleph_\omega))\prec L_\beta({\cal P}(\aleph_\omega)).
		 \end{equation*}
As before, this is equivalent to $L_\alpha({\cal P}(\aleph_\omega))^{V[G]}\prec L_\beta({\cal P}(\aleph_\omega))^{V[G]}$. Let $X={\cal P}(\aleph_\omega)$. As $X\in L_\alpha({\cal P}(\aleph_\omega))^{V[G]}$, by absoluteness of constructibility and Lemma \ref{level}, 
\begin{equation*}
		 L_\alpha(X)=(L(X))^{L_\alpha({\cal P}(\aleph_\omega)[G]},
		\end{equation*}
	and the same holds for $\beta$. So for any formula $\varphi$ and any $a_1,\dots,a_n\in L_\alpha(X)$,
	\begin{multline*}
	 L_\alpha(X)\vDash\varphi(a_1,\dots,a_n)\text{ iff}\\
	 \text{iff } L_\alpha({\cal P}(\aleph_\omega)[G]\vDash(L(X)\vDash\varphi(a_1,\dots,a_n))\\
	 \text{iff } L_\beta({\cal P}(\aleph_\omega)[G]\vDash(L(X)\vDash\varphi(a_1,\dots,a_n))\\
	 \text{iff } L_\beta(X)\vDash\varphi(a_1,\dots,a_n).
	\end{multline*}
\end{proof}
	
	All the other requirements have the effect of making the generic I0 more similar to the real one: for example, if $j:L(V_{\lambda+1})\prec L(V_{\lambda+1})$ then it must be that $j(\lambda)=\lambda$, and this is a fundamental point for I0, and because of point (2) $j(\aleph_\omega)=\aleph_\omega$ (this will also be crucial to prove the J\'{o}nsson-ness of key cardinals). Point (1) gives even more:
	
	\begin{lem}
\label{theta}
 Suppose generic I0 at $\aleph_\omega$. Then $\Theta^{V[G]}=\Theta$ and $j(\Theta)=\Theta$.
\end{lem}
\begin{proof}
  Let $X={\cal P}(\aleph_\omega)$. Suppose not, so there exists a $p\in\mathbb{P}$ and a name $\tau$ such that $p\vDash\tau:\hat{X}\to\hat{\Theta}$. Then define (in $V$) $e:\mathbb{P}\times{\cal P}(\aleph_\omega)\to\Theta$ as $e(q,a)=\alpha$ iff $q\vDash\tau(a)=\alpha$. But then $e$ is a surjection: for every $\beta<\Theta$ there exists $a\in{\cal P}(\aleph_\omega)$ and $q\leq p$ such that $q\vDash\tau(a)=\beta$. By Definition \ref{genericI0}(1), this induces a surjection $e^+:{\cal P}(\aleph_\omega)\twoheadrightarrow\Theta$, contradiction.
\end{proof}
  
 I0 is a strengthening of I1, that is $j:V_{\lambda+1}\prec V_{\lambda+1}$: as already noticed, point (3) of the definition reflects this, as it implies that $j({\cal P}(\aleph_\omega))={\cal P}(\aleph_\omega)^{L({\cal P}(\aleph_\omega))[G]}={\cal P}(\aleph_\omega)[G]$, and therefore $j\upharpoonright {\cal P}(\aleph_\omega):{\cal P}(\aleph_\omega)\prec {\cal P}(\aleph_\omega)[G]$, that is a nice definition for generic I1. Finally, point (5) is there to give the necessary amenability for $j$, that in I0 is for free:
  
  \begin{lem}
  \label{amenability}
  Suppose that generic I0 holds at $\aleph_\omega$ and let $G$ and $j$ witness that. Then for unbounded (for all) $\alpha<\Theta$, $j\upharpoonright L_\alpha({\cal P}(\aleph_\omega))\in L({\cal P}(\aleph_\omega))[G]$.
 \end{lem}
 \begin{proof}
  By Lemma \ref{ubdgood}, we can assume that $\alpha$ is good. Then $j$ is defined only by its behavior on ${\cal P}(\aleph_\omega)$, that in turn is defined by its behavior on $H(\aleph_\omega)$, because 
	\begin{equation*}
	  j(a)=j(\bigcup_{n\in\omega}(a\cap\aleph_n))=\bigcup_{n\in\omega}j(a\cap\aleph_n).
	\end{equation*}
	 By Definition \ref{genericI0}(5) $j\upharpoonright H(\aleph_\omega)\in L({\cal P}(\aleph_\omega))[G]$, and therefore this is true also for $j\upharpoonright L_\alpha({\cal P}(\aleph_\omega))$.
 \end{proof}
  
\section{Consequences of generic I0}

 Let us recall the theorem that summarizes the consequences of generic I0 at $\aleph_\omega$:

 \begin{teo}
 \label{thm:main}
  Suppose that generic I0 holds at $\aleph_\omega$ as witnessed by a forcing notion $\mathbb{P}$. Then $\aleph_\omega$ is J\'{o}nsson. Moreover, in $L({\cal P}(\aleph_\omega))$ (see below for all the definitions):
	\begin{itemize}
	 \item if $\aleph_{\omega+1}^{V[G]}=\aleph_{\omega+1}$, then $\aleph_{\omega+1}$ is $\omega$-strongly measurable and J\'{o}nsson;
	 \item $\Theta$ is ``inaccessible'', in the sense that it is regular and if $\alpha<\Theta$ then there is a surjection from ${\cal P}(\aleph_\omega)$ to ${\cal P}(\alpha)$;
	 \item if $\mathbb{P}$ is $\omega$-closed, then $\Theta$ is limit of $\omega$-strongly measurable cardinals that are measurable.
	\end{itemize}
 \end{teo}

 By Remark \ref{jonsson} and Lemma \ref{amenability} it is immediate to see that $\aleph_\omega$, and $\aleph_{\omega+1}$ when it is a fixed point, are J\'{o}nsson in $L({\cal P}(\aleph_\omega))$ . But the definition of $\aleph_\omega$ J\'{o}nsson involves only elements in ${\cal P}(\aleph_\omega)$, therefore $\aleph_\omega$ is J\'{o}nsson also in $V$.

  \begin{prop}
  \label{measure}
   Suppose that I0 holds at $\aleph_\omega$, witnessed by $G,\ j$. If $\aleph_{\omega+1}^{V[G]}=\aleph_{\omega+1}$, then $\aleph_{\omega+1}$ is $\omega$-strongly measurable in $L({\cal P}(\aleph_\omega))$.
  \end{prop}
  
  It will be witnessed by $\gamma=\crt(j)$. As $L({\cal P}(\aleph_\omega))\vDash \DC_{\aleph_\omega}$ by Remark \ref{DC}, this will imply that $\aleph_{\omega+1}$ is measurable in $L({\cal P}(\aleph_\omega))$ by Remark \ref{stronglyornot}. 
  
  \begin{proof}
   We prove that $E^{\aleph_{\omega+1}}_\omega$ cannot be partitioned in $\crt(j)$ stationary sets. As $\aleph_\omega$ is strong limit, this suffices to prove the Proposition. Suppose not, and fix $\langle S_\alpha:\alpha<\crt(j)\rangle$ a partition of $E^{\aleph_{\omega+1}}_\omega$ in $\crt(j)$ stationary sets. Then 
	\begin{equation*}
	 j(\langle S_\alpha:\alpha<\crt(j)\rangle)=\langle T_\alpha:\alpha<j(\crt(j))\rangle
	\end{equation*}
	 is a partition of $(E^{\aleph_{\omega+1}}_\omega)^{V[G]}$ into $j(\crt(j))$ stationary sets. Consider $D=\{\alpha<\aleph_{\omega+1}:j(\alpha)=\alpha\}$. By Lemma \ref{amenability}, as $j\upharpoonright L_{\aleph_{\omega+1}}({\cal P}(\aleph_\omega))\in L({\cal P}(\aleph_\omega))[G]$, $D\in L({\cal P}(\aleph_\omega))[G]$. As ${\cal P}^{V[G]}(\aleph_\omega)$ can be considered closed under $\aleph_\omega$-sequences, $\aleph_{\omega+1}^{V[G]}$ is regular, and therefore $D$ is an $\omega$-club of $\aleph_{\omega+1}=\aleph_{\omega+1}^{V[G]}$, so there exists $\eta\in D\cap T_{\crt(j)}$. But there must be some $\alpha$ such that $\eta\in S_\alpha$, and so $\eta=j(\eta)\in T_{j(\alpha)}$, contradiction.
  \end{proof}
	
	The rest of Theorem \ref{thm:main} deals with ordinals under $\Theta$. The following lemma will be therefore very useful:
	
	\begin{prop}
 \label{cofinalfixed}
 $\{\alpha<\Theta:j(\alpha)=\alpha\}$ is cofinal in $\Theta$.
\end{prop}
\begin{proof}
 Let 
 \begin{equation*}
  U=\{X\subseteq {\cal P}(\aleph_\omega): X\in L({\cal P}(\aleph_\omega)),\ (H(\aleph_\omega),j\upharpoonright H(\aleph_\omega))\in j(X)\}.
 \end{equation*}
 It is a $L({\cal P}(\aleph_\omega))$-ultrafilter. By Definition \ref{genericI0}(5) $U$ is definable in $L({\cal P}(\aleph_\omega))[G]$, as a class, in the sense that given an $X\subseteq{\cal P}(\aleph_\omega)$, $L({\cal P}(\aleph_\omega))[G]$ ``knows'' whether $X\in U$ or not, but $U\notin L({\cal P}(\aleph_\omega))[G]$. The proof works as Lemma 5 in \cite{Woodin}, in that we prove that the embedding $j$ is close enough to a (generic) ultrapower embedding, that is easier to work with. The main innovation in \cite{Woodin} is the proof of \L{}os' Theorem: in $L(V_{\lambda+1})$ there is no Axiom of Choice, so it must be proven from scratch. Here the situation is the same, but the proof of \L{}os' Theorem is slightly different, as here the same ``reflection trick'' is not possible, and we must create one ad hoc.

The obstacle to prove \L{}os' Theorem without \AC{} is in proving that  if 
\begin{equation*}
 \{z\in{\cal P}(\aleph_\omega):L({\cal P}(\aleph_\omega))\vDash\exists y\ \varphi(f_1(x),\dots,f_n(x))\}\in U, 
\end{equation*}
then there exists a function $g:{\cal P}(\aleph_\omega)\to L({\cal P}(\aleph_\omega))$ such that 
\begin{equation*}
 \{x\in V_{\lambda+1}:L({\cal P}(\aleph_\omega))\vDash \varphi(g(x),f_1(x),\dots,f_n(x))\}\in U. 
\end{equation*}
This is equivalent to prove that for each $f:{\cal P}(\aleph_\omega)\to L({\cal P}(\aleph_\omega))\setminus\emptyset$ there exists a $g:{\cal P}(\aleph_\omega)\to L({\cal P}(\aleph_\omega))\setminus\emptyset$ such that $\{x\in{\cal P}(\aleph_\omega):g(x)\in f(x)\}\in U$. Since in $L({\cal P}(\aleph_\omega))$ everything is definable from an ordinal and elements in ${\cal P}(\aleph_\omega)$, it suffices to prove it for $f$ such that for any $a\in{\cal P}(\aleph_\omega)$, $f(a)\subseteq{\cal P}(\aleph_\omega)$. For more details about this procedure, see Lemma 5.11 in \cite{Dimonte0}.

\begin{claim}
 \label{allinimage}
 For any $b\in{\cal P}(\aleph_\omega)$, there exists a $g:{\cal P}(\aleph_\omega)\to L({\cal P}(\aleph_\omega))\setminus\emptyset$ such that $j(g)((H(\aleph_\omega),j\upharpoonright H(\aleph_\omega)))=b$.
\end{claim}

Note that $(H(\aleph_\omega),j\upharpoonright H(\aleph_\omega))$ can be coded as an element of ${\cal P}(\aleph_\omega)^{V[G]}$, therefore the claim is well-formulated.

\begin{proof}[Proof of Claim]
 Let $N\subseteq H(\aleph_\omega)$. Then we define ${\cal P}(\aleph_\omega)_N$ in the same way we define ${\cal P}(\aleph_\omega)$ from $H(\aleph_\omega)$: 
\begin{equation*}
 {\cal P}(\aleph_\omega)_N=\{\bigcup_{n\in\omega}a_n:(\forall n\ a_n\subseteq\aleph_n\wedge a_n\in N)\wedge\forall n\leq m\ a_n=a_m\cap\aleph_n\}. 
\end{equation*}

As in Lemma \ref{amenability}, if we have a $k:N\prec H(\aleph_\omega)$, then it naturally extends to $k_N:{\cal P}(\aleph_\omega)_N\to {\cal P}(\aleph_\omega)$, with $k_N(\bigcup_{n\in\omega}a_n)=\bigcup_{n\in\omega}k(a_n)$. If we do not known anything else about $k$, it could be the case that $k_N$ is not elementary.

 Let 
\begin{equation*}
 g(x)=\begin{cases}
	c & \text{if }x=(N,k),\ N\subseteq H(\aleph_\omega),\ \exists Z\subseteq {\cal P}(\aleph_\omega)_N\ k_N\upharpoonright Z:Z\prec {\cal P}(\aleph_\omega),\ k(c)=b\\
	0 & \text{otherwise}.
	\end{cases}
\end{equation*}

Then
\begin{equation*}
 j(g)(x)=\begin{cases}
	c & \text{if }x=(N,k),\ N\subseteq H^{V[G]}(\aleph_\omega),\ \exists Z\subseteq ({\cal P}(\aleph_\omega)_N)^{V[G]}\\
	  & k_N\upharpoonright Z:Z\prec {\cal P}^{V[G]}(\aleph_\omega),\ k(c)=j(b)\\
	0 & \text{otherwise}.
	\end{cases}
\end{equation*}

If $\mathbb{P}$ is $\omega$-closed, then $({\cal P}(\aleph_\omega)_{H(\aleph_\omega)})^{V[G]}={\cal P}(\aleph_\omega)$, and therefore it is immediate to see that $(H(\aleph_\omega),j\upharpoonright H(\aleph_\omega))$ satisfies the conditions in the definition of $j(g)$, letting $Z=({\cal P}(\aleph_\omega)_N)^{V[G]}$, and therefore $j(g)((H(\aleph_\omega),j\upharpoonright H(\aleph_\omega)))=b$, because $j$ is injective. If $\mathbb{P}$ is not $\omega$-closed, then possibly in ${\cal P}(\aleph_\omega)_{H(\aleph_\omega)}$ calculated in $V[G]$ more elements can appear than the ones in ${\cal P}(\aleph_\omega)$, as there can be, in $V[G]$, $\omega$-sequences of elements in $H(\aleph_\omega)$ that are not in $V$. But anyway ${\cal P}(\aleph_\omega)\in L({\cal P}(\aleph_\omega))[G]$, so we can just let $Z={\cal P}(\aleph_\omega)\subseteq ({\cal P}(\aleph_\omega)_{H(\aleph_\omega)})^{V[G]}$, and then again $j(g)((H(\aleph_\omega),j\upharpoonright H(\aleph_\omega)))=b$.
\end{proof}

In particular, $b$ can be an element of $j(f)((H(\aleph_\omega),j\upharpoonright H(\aleph_\omega)))$, therefore \L{}os' Theorem is proved.

The ultrapower construction therefore does yield an elementary embedding, we call it $k$, and the following diagram commutes:
\newline
\begin{center}
\begin{tikzpicture}
 \node (a) at (0,3) {$L({\cal P}(\aleph_\omega))$};
 \node (b) at (3,3) {$L({\cal P}(\aleph_\omega))[G]$};
 \node (c) at (3,0) {$\Ult(L({\cal P}(\aleph_\omega)),U)$};
 \draw [->] (a) -- (b) node[anchor=south, midway] {$j$};
 \draw [->] (a) -- (c) node[anchor=west, midway] {$k$};
 \draw [->] (c) -- (b) node[anchor=west, midway] {$h$};
\end{tikzpicture}
\end{center}

where $h([f])=j(f)((H(\aleph_\omega),j\upharpoonright H(\aleph_\omega)))$. By Claim \ref{allinimage}, $h"\Ult(L({\cal P}(\aleph_\omega)),U)\supseteq{\cal P}(\aleph_\omega)$, and therefore the critical point of $h$ is bigger than $\Theta$, because the critical point of $h$ is not in the image of $h$. Therefore $j(\Theta)=k(\Theta)$ and $j\upharpoonright L_\Theta({\cal P}(\aleph_\omega))=k\upharpoonright L_\Theta({\cal P}(\aleph_\omega))$. Also, by elementarity, $\Ult(L({\cal P}(\aleph_\omega)),U)\vDash V=L(k({\cal P}(\aleph_\omega)))$. As $k({\cal P}(\aleph_\omega))=j({\cal P}(\aleph_\omega))={\cal P}^{V[G]}(\aleph_\omega)$, this means that $\Ult(L({\cal P}(\aleph_\omega)),U)=L({\cal P}^{V[G]}(\aleph_\omega))=L({\cal P}(\aleph_\omega))[G]$.

Now it is the almost the same as \cite{Woodin}. It is standard to prove that every strong limit cardinal with large cofinality is a fixed point of $k$, so the class of fixed points $I$ is cofinal in $\Ord$, and therefore every element of $L({\cal P}(\aleph_\omega))$ is defined from elements of $I\cup {\cal P}(\aleph_\omega)$: if not, the inverse of the collapse of the Skolem closure of $I\cup {\cal P}(\aleph_\omega)$ would be a non-trivial elementary embedding, but the critical point would be so large that it would be in $I$, and that is a contradiction. This argument would hold also in $L({\cal P}^{V[G]}(\aleph_\omega))$, so any element in $L({\cal P}^{V[G]}(\aleph_\omega))$ is defined from elements of $I\cup {\cal P}^{V[G]}(\aleph_\omega)$.

So any ordinal $\beta<\Theta$ is definable from some elements of $I\cup {\cal P}(\aleph_\omega)$. Pick $\beta$ so that $\mathbb{P},\ {\cal P}(\aleph_\omega),\ G\in L_\beta({\cal P}(\aleph_\omega))^{V[G]}$. Let $i_0,\dots,i_n\in I$ such that $\beta$ and $\mathbb{P}$ are definable from ${\cal P}(\aleph_\omega)\cup\{i_0,\dots,i_n\}$. Let $Z$ be the Skolem closure of $\{i_0,\dots,i_n\}\cup{\cal P}(\aleph_\omega)$. In particular $\beta,\ \mathbb{P}\in Z$. Let $\alpha=\sup Z$, so that $\beta<\alpha$.

By the lifting lemma and the fact that $\mathbb{P}\in Z$ we have that $Z[G]\prec L({\cal P}(\aleph_\omega))[G]=L({\cal P}^{V[G]}(\aleph_\omega))$, and also $Z[G]$ contains ${\cal P}(\aleph_\omega)[G]={\cal P}^{V[G]}(\aleph_\omega)$ and $\{i_0,\dots,i_n\}$. But, by elementarity, $k(Z)$ is the Skolem closure of $\{i_0,\dots,i_n\}\cup{\cal P}^{V[G]}(\aleph_\omega)$ in $L({\cal P}^{V[G]}(\aleph_\omega))$, therefore $k(Z)\subseteq Z[G]$. But then $k(\alpha)=\sup(k(Z))\leq \sup Z[G]\leq\sup Z=\alpha$, so $k(\alpha)=\alpha$.

Therefore the fixed points of $k$ are cofinal in $\Theta$. As 
\begin{equation*}
 j\upharpoonright L_\Theta({\cal P}(\aleph_\omega))=k\upharpoonright L_\Theta({\cal P}(\aleph_\omega)), 
\end{equation*}
also the fixed points of $j$ are cofinal in $\Theta$.
\end{proof}
  
  \begin{teo}
	  \label{teo2}
    $\Theta$ is ``inaccessible'' in $L({\cal P}(\aleph_\omega))$, in the sense that it is regular and if $\alpha<\Theta$ then there is a surjection from ${\cal P}(\aleph_\omega)$ to ${\cal P}(\alpha)$.
  \end{teo}

  \begin{lem}[Weak Coding Lemma]
  \label{WCL}
   Suppose generic I0 for $\aleph_\omega$, witnessed by $j$, $\mathbb{P}$. Then in $L({\cal P}(\aleph_\omega))$, for all $\kappa<\Theta$ and for all $\pi:{\cal P}(\aleph_\omega)\twoheadrightarrow\kappa$ $\exists\gamma<\Theta$ such that $\forall X\subseteq{\cal P}(\aleph_\omega)$ if $\pi''X$ is cofinal in $\kappa$, then $\exists Y\subseteq X$, $Y\in L_{\gamma}({\cal P}(\aleph_\omega))$ such that $\pi''Y$ is cofinal in $\kappa$.
  \end{lem}
	
	\begin{proof}
	 We introduce some notation, to make the proof more readable.
	
	 If $X\subseteq{\cal P}(\aleph_\omega)$, $X\in L({\cal P}(\aleph_\omega))$, we say that $X$ is $\pi$-unbounded (in $\kappa$) if $\pi''X$ is cofinal in $\kappa$. So the Weak Coding Lemma says that for any $\pi$ surjection of ${\cal P}(\aleph_\omega)$ to some ordinal there exists a $\gamma$ such that any $\pi$-unbounded set has a subset that is in $L_{\gamma}({\cal P}(\aleph_\omega))$ (therefore less complex) that it is still $\pi$-unbounded. 
	
	Suppose not, i.e., there exist $\kappa$, $\pi:{\cal P}(\aleph_\omega)\twoheadrightarrow\kappa$, such that $\forall\gamma<\Theta\ \exists X\subseteq{\cal P}(\aleph_\omega)$ $\pi$-unbounded such that $\forall Y\subseteq X$, $Y\in L_\gamma({\cal P}(\aleph_\omega))$ $Y$ is $\pi$-bounded. We just write NWCL$^{\kappa}$ to indicate this and underline the role of $\kappa$.
	
	Now we want to fix $\kappa$.
	
	\begin{claim}
\label{newY}
 Let $\pi:{\cal P}(\aleph_\omega)\twoheadrightarrow\kappa$, let $X\subseteq{\cal P}(\aleph_\omega)$ $\pi$-unbounded. Then if there exists a $Y\in L_\gamma({\cal P}(\aleph_\omega))[G]$, with $\gamma$ limit and $\mathbb{P}\in L_\gamma({\cal P}(\aleph_\omega))[G]$, such that $Y$ is $\pi$-unbounded and $Y\subseteq X$, then there exists $Y^*\in L_\gamma({\cal P}(\aleph_\omega))$ that satisfies the same.
\end{claim}
\begin{proof}
 Let $p\in G$ and $\tau\in L_\gamma({\cal P}(\aleph_\omega))$ such that $p\vDash \tau=Y\wedge Y\subseteq X$. Define 
 \begin{equation*}
  Y^*=\{a\in{\cal P}(\aleph_\omega):\exists q\leq p\ q\vDash a\in\tau\}. 
 \end{equation*}
Then $Y^*\in L_\gamma({\cal P}(\aleph_\omega))$ and $Y\subseteq Y^*\subseteq X$, therefore $\pi''Y^*\supseteq\pi''Y$ is cofinal in $\kappa$. 
\end{proof}

An immediate consequence of Claim \ref{newY} is that if NWCL$^\kappa$ does not hold in $L({\cal P}(\aleph_\omega))[G]$, then it does not hold directly in $L({\cal P}(\aleph_\omega))$: Suppose NWCL$^\kappa$ does not hold in $L({\cal P}(\aleph_\omega))[G]$ and it holds in $L({\cal P}(\aleph_\omega))$. Then there exists $\pi:{\cal P}(\aleph_\omega)\twoheadrightarrow\kappa$, such that $\forall\gamma<\Theta\ \exists X\subseteq{\cal P}(\aleph_\omega)$ $\pi$-unbounded such that $\forall Y\subseteq X$, $Y\in L_\gamma({\cal P}(\aleph_\omega))$, $Y$ is $\pi$-bounded. Fix one of these $\pi$. We can trivially extend $\pi$ to $\pi^*:{\cal P}^{V[G]}(\aleph_\omega)\twoheadrightarrow\kappa$. Then, in $L({\cal P}(\aleph_\omega))[G]$, there exists a $\gamma$ such that all the $\pi^*$-unbounded sets have a $\pi^*$-unbounded subset in $L_\gamma({\cal P}(\aleph_\omega))^{V[G]}$. We can choose $\gamma$ so that it satisfies Lemma \ref{level}. Let $X\subseteq {\cal P}(\aleph_\omega)$ in $L({\cal P}(\aleph_\omega))$ that is $\pi$-unbounded but such that all its subsets in $L_\gamma({\cal P}(\aleph_\omega))$ are $\pi$-bounded. Then, in $L_\gamma({\cal P}(\aleph_\omega))[G]$, $X$ is also $\pi^*$-unbounded, so there is a $Y\subseteq X$, $Y\in L_\gamma({\cal P}(\aleph_\omega))^{V[G]}=L_\gamma({\cal P}(\aleph_\omega))[G]$ $\pi^*$-unbounded and therefore $\pi$-unbounded. Contradiction by Claim \ref{newY}.

  Now let $\kappa_0$ be the least such that NWCL$^{\kappa_0}$ holds.
	
	\begin{claim}
 $j(\kappa_0)=\kappa_0$.
\end{claim}
\begin{proof}
 Suppose not, then $\kappa_0<j(\kappa_0)$. As $j(\kappa_0)$ is minimal in $L({\cal P}(\aleph_\omega))[G]$, this means that NWCL$^{\kappa_0}$ does not hold in $L({\cal P}(\aleph_\omega))[G]$. But then by Claim \ref{newY} it does not hold in $L({\cal P}(\aleph_\omega))$. Contradiction.
\end{proof}

Now let us drop the superscript and write NWCL instead of NWCL$^{\kappa_0}$. The free parameters in NWCL are $\omega$ and $\kappa_0$, all fixed points of $j$. By definition of $\kappa_0$, NWCL is true. As all the sets involved are subsets of ${\cal P}(\aleph_\omega)$, or codeable as such, NWCL is true in $L_\Theta({\cal P}(\aleph_\omega))$.

Let 
\begin{equation*}
 \Gamma=\{\alpha<\Theta:L_\alpha({\cal P}(\aleph_\omega))\prec L_\Theta({\cal P}(\aleph_\omega))\}. 
\end{equation*}
Then by Lemma \ref{lifting} $\Gamma$ and $\Gamma^{L({\cal P}(\aleph_\omega))[G]}$ have the same tail. Using Lemma \ref{cofinalfixed}, fix a $\delta<\Theta$, $\delta>\kappa_0$  that is a fixed point of $j$, such that $j\upharpoonright H(\aleph_\omega)\in L_\delta({\cal P}(\aleph_\omega))^{V[G]}$, and so that 
\begin{equation*}
 \Gamma\cap[\delta,\Theta]=\Gamma^{L({\cal P}(\aleph_\omega))[G]}\cap[\delta,\Theta]=\Gamma'. 
\end{equation*}
Fix an enumeration of $\Gamma'=\langle\gamma_\alpha:\alpha<\Theta\rangle$. Note that $\Gamma'=j(\Gamma')$ and $j(\gamma_\alpha)=\gamma_{j(\alpha)}$. 

As NWLC is true in $L_\Theta({\cal P}(\aleph_\omega))$, by elementarity it is also true in all $L_{\gamma}({\cal P}(\aleph_\omega))$, for $\gamma\in\Gamma'$. Therefore NWLC is true in $L_{\gamma_{\crt(j)+1}}({\cal P}(\aleph_\omega))$. Fix $\pi:{\cal P}(\aleph_\omega)\twoheadrightarrow\kappa_0$, and let $X\subseteq{\cal P}(\aleph_\omega)$ be $\pi$-unbounded and such that all $Y\subseteq X$, $Y\in L_{\gamma_{\crt(j)}}({\cal P}(\aleph_\omega))$ are $\pi$-bounded, with $\pi, X\in L_{\gamma_{\crt(j)+1}}({\cal P}(\aleph_\omega))$. Then $j(X)$ is $j(\pi)$-unbounded, and all the $Y\subseteq X$, $Y\in L_{\gamma_{j(\crt(j))}}({\cal P}(\aleph_\omega))^{V[G]}$ are $j(\pi)$-bounded. But consider $j``X$: it is a subset of $j(X)$, it is defined from $X$ and $j\upharpoonright H(\aleph_\omega)$, and therefore is in $L_{\gamma_{\crt(j)+1}+1}({\cal P}(\aleph_\omega))^{V[G]}$. As $\crt(j)+1<j(\crt(j))$, $\gamma_{\crt(j)+1}<\gamma_{j(\crt(j))}$ and therefore $\gamma_{\crt(j)+1}+1\leq\gamma_{j(\crt(j))}$, so $j``X\in L_{\gamma_{j(\crt(j))}}({\cal P}(\aleph_\omega))^{V[G]}$. But $j''\pi''X= j(\pi)''(j''X)$ and $j''\kappa_0$ is cofinal in $\kappa_0$, therefore $j``X$ is $j(\pi)$-unbounded, contradiction.
\end{proof}

\begin{lem}[Coding Lemma]
  Suppose generic I0 for $\aleph_\omega$, witnessed by $\mathbb{P}$. Then in $L({\cal P}(\aleph_\omega))$, for all $\kappa<\Theta$ and for all $\pi:{\cal P}(\aleph_\omega)\twoheadrightarrow\kappa$ $\exists\gamma_0$ such that $\forall X\subseteq{\cal P}(\aleph_\omega)$ $\exists Y\subseteq X$, $Y\in L_{\gamma_0}({\cal P}(\aleph_\omega))$ such that $\pi''X=\pi''Y$.
\end{lem}
\begin{proof}
 Proving the Coding Lemma from the Weak Coding Lemma is standard (see for example the proof of Lemma 22 in \cite{Woodin}, pages 149--150), and does not need generic I0.

 Let $\pi:{\cal P}(\aleph_\omega)\twoheadrightarrow\kappa$, and suppose by induction that the Coding Lemma holds for all $\pi':{\cal P}(\aleph_\omega)\twoheadrightarrow\alpha$ with $\alpha<\kappa$. Let $\tau$ be a bijection between ${\cal P}(\aleph_\omega)$ and ${\cal P}(\aleph_\omega)\times{\cal P}(\aleph_\omega)$, and let $\pi^*:{\cal P}(\aleph_\omega)\twoheadrightarrow\kappa$ be $\pi\circ\tau_1$ (therefore if $\tau(x)=(a,b)$ then $\pi^*(x)=\pi(a)$). For any $\alpha<\kappa$, let $\pi_{\alpha}:{\cal P}(\aleph_\omega)\twoheadrightarrow\alpha$ be
\begin{equation*}
	\pi_\alpha(x)=\begin{cases}
		\pi(x) & \text{if }\pi(x)<\alpha\\
		0      & \text{otherwise}
	\end{cases}
\end{equation*}
Let $\gamma_\alpha$ witness the Coding Lemma for $\pi_\alpha$ and fix a $\gamma_\kappa$ that witnesses the Weak Coding Lemma for $\pi^*$. Let $\beta_0=\sup_{\alpha\leq\kappa}\gamma_\alpha$: As $\Theta$ is regular, $\beta_0<\Theta$, therefore there exists $\rho:{\cal P}(\aleph_\omega)\twoheadrightarrow L_{\beta_0}({\cal P}(\aleph_\omega))$. Let $\rho,\tau\in L_\beta({\cal P}(\aleph_\omega))$. Then we claim that $\beta+1$ witnesses the Coding Lemma for $\pi$.

Let $X\subseteq{\cal P}(\aleph_\omega)$. Consider 
\begin{equation*}
 A=\{x\in{\cal P}(\aleph_\omega):\rho(\tau_2(x))\text{ witnesses the Coding Lemma for }\pi_{\pi^*(x)},\ X\}.
\end{equation*}
 In other words, $A$ is the set of all $(\alpha,Y)$ such that $Y$ satisfies the Coding Lemma for $\pi_\alpha$ and $X$, but coded as a subset of ${\cal P}(\aleph_\omega)$. Since by induction for every $\alpha<\kappa$ there is a $Y\in L_{\beta_0}({\cal P}(\aleph_\omega))$ that satisfies the Coding Lemma for $\pi_\alpha$ and $X$, $\pi^*[A]=\kappa$, and by the Weak Coding Lemma there exists a $B\in L_{\gamma_\kappa}({\cal P}(\aleph_\omega))$, $B\subseteq A$ such that $B$ is $\pi^*$-unbounded. Then let $Y=\bigcup\{\rho(\tau_2(x)):x\in B\}$. Clearly $Y\in L_{\beta+1}({\cal P}(\aleph_\omega))$. 

Suppose there exists an $x\in X$ such that $\pi(x)=\gamma$. We must find a $y\in Y$ such that $\pi(y)=\gamma$. As $B$ is $\pi^*$-unbounded, there exists an $\alpha>\gamma$ and a $z\in B$ such that $\pi^*(z)=\alpha$, and therefore $\pi(x)=\pi_\alpha(x)$. As $B\subseteq A$, this means that $\rho(\tau_2(z))$ witnesses the Coding Lemma for $\pi_{\pi^*(z)}=\pi_\alpha$ and $X$. Therefore $\pi_\alpha''X=\pi_\alpha''\rho(\tau_2(z))$ and there exists a $y\in\rho(\tau_2(z))$ such that $\pi(y)=\pi_\alpha(y)=\gamma$.  
\end{proof}

\begin{proof}[Proof of Theorem \ref{teo2}]
 Let $\alpha<\Theta$ and $A\subseteq\alpha$. Fix a $\pi:{\cal P}(\aleph_\omega)\twoheadrightarrow\alpha$, $\pi\in L_\beta({\cal P}(\aleph_\omega))$ and let $\gamma\geq\beta$ witness the Coding Lemma for $\pi$. Let $X=\pi^{-1}[A]$. Then, by the Coding Lemma, there exists $Y\in L_{\gamma}({\cal P}(\aleph_\omega))$, $Y\subseteq X$ such that $\pi''Y=\pi''X=A$. But then $A\in L_\gamma({\cal P}(\aleph_\omega))$. 

Therefore ${\cal P}(\alpha)\subseteq L_\gamma({\cal P}(\aleph_\omega))$, and as there is a surjection from ${\cal P}(\aleph_\omega)$ to $L_\gamma({\cal P}(\aleph_\omega))$, the theorem is proved.
\end{proof}

\begin{teo}
\label{Thetalimit}
 If $\mathbb{P}$ is $\omega$-closed, then $\Theta$ is limit of $\omega$-strongly measurable cardinals that are measurable.
\end{teo}

\begin{proof}
Fix an $\alpha$ such that $j(\alpha)=\alpha$ and such that Lemma \ref{level} holds, and let $\kappa$ be the least ordinal bigger than $\alpha$ such that $L_\kappa({\cal P}(\aleph_\omega))\prec_1 L_\Theta({\cal P}(\aleph_\omega))$. Then, using Lemma\ref{lifting}, $j(\kappa)=\kappa$. We will prove that $\kappa$ is an $\omega$-strongly measurable cardinal and that is measurable, and by Lemma \ref{cofinalfixed} the arbitrarity of $\alpha$ will prove the theorem,

Define $\Theta^{L_\kappa({\cal P}(\aleph_\omega))}$ as the supremum of the ordinals $\alpha$ such that there is a surjection $\rho:{\cal P}(\aleph_\omega)\twoheadrightarrow\alpha$ such that $\{(a,b):\rho(a)<\rho(b)\}\in L_\kappa({\cal P}(\aleph_\omega))$. By elementarity $\Theta^{L_\kappa({\cal P}(\aleph_\omega))}=\kappa$. It is now tempting to dump all the information we have about $\Theta$ on $\kappa$, for example that $\Theta$ is regular or the Coding Lemma itself. But one must be cautious: $\Theta$ is not an element of $L_\Theta({\cal P}(\aleph_\omega))$, therefore not all properties can be reflected (not regularity, for example), and the Coding Lemma has too high complexity. But it is possible to prove both, thanks to the analysis of stable ordinals that has been brought forward for $L(\mathbb{R})$.

Like in $L(\mathbb{R})$, there exists a $\Sigma_1(\alpha)$ partial map from ${\cal P}(\aleph_\omega)$ to $L_\kappa({\cal P}(\aleph_\omega))$: As $L_\kappa({\cal P}(\aleph_\omega))\vDash V=\HOD_{{\cal P}(\aleph_\omega)}$, it is possible to build $\Sigma_1$ partial Skolem functions inside $L_\kappa({\cal P}(\aleph_\omega))$, defining for any $\Sigma_1$ formula and any $a\in{\cal P}(\aleph_\omega)$ $h_{\phi,a}(x)$ as the smallest element in $\OD_a$ that satisfies $\varphi(x)$. The closure of ${\cal P}(\aleph_\omega)\cup\alpha$ under such Skolem functions is a $\Sigma_1$ elementary substructure of $L_\Theta({\cal P}(\aleph_\omega))$, and by condensation its collapse is some $L_\gamma({\cal P}(\aleph_\omega))$, with $\gamma>\alpha$. But $\kappa$ was the least one, so $\gamma=\kappa$. Therefore every element of $L_\kappa({\cal P}(\aleph_\omega))$ is image of the collapse of some Skolem function, and coding $\varphi,\ a$ and $x$ in a unique element of ${\cal P}(\aleph_\omega)$ we have the $\Sigma_1(\alpha)$ partial surjection. 

In the same way, more carefully, it is possible for any $\beta<\kappa$ to find in $L_\kappa({\cal P}(\aleph_\omega))$ a total surjection from ${\cal P}(\aleph_\omega)$ to $L_\beta({\cal P}(\aleph_\omega))$. Therefore, $\kappa$ is actually the supremum of the $\mathbf{\Delta}^2_1(\alpha)$ prewellorderings of ${\cal P}(\aleph_\omega)$ (the proof is the same as in Lemma 1.11 and Lemma 1.12 in \cite{Steel}). 

\begin {lem}
 \label{boundedness}
 Let $\rho:{\cal P}(\aleph_\omega)\to L_\kappa({\cal P}(\aleph_\omega))$ be a $\Sigma_1(\alpha)$ partial map. Then for every $Z\subseteq\dom(\rho)$ such that $Z\in L_\kappa({\cal P}(\aleph_\omega))$, $\rho\upharpoonright Z$ is bounded.
\end {lem}
\begin{proof}
Let $Z\subseteq\dom(\rho)$, $Z\in L_\kappa({\cal P}(\aleph_\omega))$. Then $Z$ is actually $\mathbf{\Delta}^2_1(\alpha)$ from ${\cal P}(\aleph_\omega)$. We can build a $\mathbf{\Delta}^2_1(\alpha)$ prewellordering on $Z$, connecting to any $z\in Z$ the least ordinal in the constructive hierarchy where $\rho(z)$ appears. Then $\rho\upharpoonright Z$ cannot be unbounded, otherwise there would be a $\mathbf{\Delta}^2_1(\alpha)$ prewellordering of length $\kappa$. 
\end{proof}

We proved the following:

\begin{lem}
 \label{boundednessofpi}
 There exists a partial $\rho:{\cal P}(\aleph_\omega)\twoheadrightarrow L_\kappa({\cal P}(\aleph_\omega))$, $\Sigma_1(\alpha)$-definable from ${\cal P}(\aleph_\omega)$, such that on every $Z\subseteq\dom(\rho)$, $Z\in L_\kappa({\cal P}(\aleph_\omega))$ we have that $\rho$ is bounded.
\end{lem}

In a similar fashion, we can prove that $\cof(\kappa)>\aleph_\omega$: Suppose that there exists a cofinal sequence of length $<\aleph_\omega$ in $\kappa$. Then using $\DC_{\aleph_\omega}$ for every element of the sequence one can choose a $\mathbf{\Delta}^2_1(\alpha)$ prewellordering with such length, and glue them all together to form a $\mathbf{\Delta}^2_1(\alpha)$ prewellordering of length $\kappa$, that is contradictory. With the same reasoning the cofinality of $\kappa$ is larger then $\aleph_\omega$ also in $L({\cal P}(\aleph_\omega))^{V[G]}$.

Now the proof that $\kappa$ is $\omega$-strongly measurable is as in Proposition \ref{measure}. Consider $E^\kappa_\omega$ and let $\langle S_\alpha:\alpha<\crt(j)\rangle$ a partition of $E^\kappa_\omega$ in stationary subsets. Let 
\begin{equation*}
	\langle T_\alpha:\alpha<j(\crt(j))\rangle=j(\langle S_\alpha:\alpha<\crt(j)\rangle). 
\end{equation*}
Let $D=\{\alpha<\kappa:\cof(\alpha)=\omega,\ j(\alpha)=\alpha\}$. Then by Definition \ref{genericI0}(5) $D\in L({\cal P}(\aleph_\omega))[G]$ and $D$ is naturally closed on $E^\kappa_\omega$. As $(\cof(\kappa)>\omega)^{L({\cal P}(\aleph_\omega)[G]}$, $D$ is unbounded. Then $D$ is in the club filter of $\kappa$, and the proof is as before.

We should prove now that $\kappa$ is measurable. Remark \ref{stronglyornot} does not help, as we do not have enough Dependent Choice to prove that the club filter $F$ on $\kappa$ is $\kappa$-complete. In the $L(V_{\lambda+1})$ case (see Lemma 22 in \cite{Woodin}) this was proved defining $F_0$ as the filter of the fixed points of the elementary embeddings from $L_\kappa(V_{\lambda+1})$ to itself: then $F_0\subseteq F$ on $E^\kappa_\omega$, with the usual proof $F_0$ is an ultrafilter on a stationary set, and it is easier to prove that $F_0$ is $\kappa$-complete.

Unfortunately, this does not work in this setting: the problem is that in the I0 case $F_0$ is both in the domain and the codomain of the $j$ that witnesses I0, and it is a fixed point, but in the generic I0 case, instead, the filter 
\begin{equation*}
 F_0=\{k:L_\kappa({\cal P}(\aleph_\omega))\prec L_\kappa({\cal P}(\aleph_\omega))[G]\}
\end{equation*}
 exists only in $V[G]$, and not in $V$. The objective, therefore, is to find a filter $F_0'\subseteq F$ in $V$ such that $F_0\subseteq j(F_0')\subseteq F$, and proving that there exists a stationary set such that $F_0'$ is an ultrafilter on it. Then the Coding Lemma will be used instead of $\DC_\kappa$ to prove that $F_0'$ is $\kappa$-complete, and this will prove that $F$ is $\kappa$-complete. The steps of the proof are therefore the following:
\begin{itemize}
	\item Proving the Coding Lemma in $L_\kappa({\cal P}(\aleph_\omega))$;
	\item Proving that there is no partition of $\kappa$ in $\crt(j)$ $F_0'$-positive sets;
	\item Proving that $F_0'\subseteq F$ on $E^\kappa_\omega$;
	\item Proving that $F_0'$ is $\kappa$-complete.
\end{itemize}

The Weak Coding Lemma in $L_\kappa({\cal P}(\aleph_\omega))$ does not follow from elementarity, but it can be re-proved in the exact same way as in Lemma \ref{WCL}: It is essential that $\Theta^{L_\kappa({\cal P}(\aleph_\omega))}=\kappa$, so that $\gamma_\alpha$ is defined, and since $\cof(\kappa)>\aleph_\omega$ the sequence of $\gamma_\alpha$ is longer than $\aleph_\omega$, and therefore the contradiction with the existence of $\gamma_{\crt(j)}$ holds.

\begin{lem}
 $\kappa$ is regular in $L({\cal P}(\aleph_\omega))$.
\end{lem}
\begin{proof}
 It is similar to the proof of $\cof(\kappa)>\aleph_\omega$, using the Weak Coding Lemma instead of $\DC_{\aleph_\omega}$. Let $\beta=\cof(\kappa)$ and pick a $f:\beta\to\kappa$ cofinal in $\kappa$. Suppose that $\beta<\kappa$. As $\Theta^{L_\kappa({\cal P}(\aleph_\omega))}=\kappa$, there exists a surjection $\pi:{\cal P}(\aleph_\omega)\twoheadrightarrow\beta$, with $\pi\in L_\kappa({\cal P}(\aleph_\omega))$. Define $\pi':{\cal P}(\aleph_\omega)\twoheadrightarrow\beta$ as $\pi'(\langle a,b\rangle)=\pi(a)$ and let $\gamma$ witness the Weak Coding Lemma in $L_\kappa({\cal P}(\aleph_\omega))$ for $\pi'$. Then let 
\begin{equation*}
 A=\{\langle a,b\rangle:b\in\dom(\rho)\wedge\rho(b)=f(\pi(a))\}.
\end{equation*}
 By the Weak Coding Lemma, there exists $B\in L_\gamma({\cal P}(\aleph_\omega))$, $B\subseteq A$ such that $B$ is $\pi'$-unbounded. By definition, this means that $\pi''(B)_0$ (the image under $\pi$ of the projection of $B$ in the first coordinate), is unbounded in $\beta$. Let 
\begin{equation*}
 Z=(B)_1=\{b:\exists a\in{\cal P}(\aleph_\omega),\ \langle a,b\rangle\in B\}. 
\end{equation*}
Then also $Z$ is in $L_\gamma({\cal P}(\aleph_\omega))$. But $f\circ\pi``(B)_0\subseteq \rho''Z$, and as $\pi''(B)_0$ is cofinal in $\beta$, $\rho``Z$ is unbounded in $\kappa$, and this contradict its boundedness as in Lemma \ref{boundednessofpi}.
\end{proof}

Again, the proof of the Coding Lemma works now in $L_\kappa({\cal P}(\aleph_\omega))$, the key points being that the Weak Coding Lemma holds in $L_\kappa({\cal P}(\aleph_\omega))$ and that $\kappa=\Theta^{L_\kappa({\cal P}(\aleph_\omega))}$ is regular (so that induction works). We have proved:

\begin{lem}[Coding Lemma in $L_\kappa({\cal P}(\aleph_\omega))$]
 \label{codingink}
 In $L_\kappa({\cal P}(\aleph_\omega))$, for all $\eta<\kappa$ and for all $\pi:{\cal P}(\aleph_\omega)\twoheadrightarrow\eta$ $\exists\gamma_0$ such that $\forall X\subseteq{\cal P}(\aleph_\omega)$, $\exists Y\subseteq X$, $Y\in L_{\gamma_0}({\cal P}(\aleph_\omega))$ such that $\pi''X=\pi''Y$.
\end{lem}

We define $F_0'$. Let $N\subseteq H(\aleph_\omega)$. Then we define ${\cal P}(\aleph_\omega)_N$ as in Proposition \ref{cofinalfixed}: 
\begin{equation*}
 {\cal P}(\aleph_\omega)_N=\{\bigcup_{n\in\omega}a_n:(\forall n\ a_n\subseteq\aleph_n\wedge a_n\in N)\wedge\forall n\leq m\ a_n\subseteq a_m\}. 
\end{equation*}
If $\kappa$ is the least such that $L_\kappa({\cal P}(\aleph_\omega)_N)\prec_1 L_{\Theta^{L({\cal P}(\aleph_\omega)_N)}}({\cal P}(\aleph_\omega)_N)$, then the situation is as before: $\kappa$ is actually the supremum of the $\mathbf{\Delta}^2_1(\alpha)$ prewellorderings of ${\cal P}(\aleph_\omega)_N$, and therefore for each $(N,k)\in {\cal E}$ we can also fix a $\rho_N$ as in Lemma \ref{boundednessofpi}, i.e., a partial surjection from ${\cal P}(\aleph_\omega)_N$ to $L_\kappa({\cal P}(\aleph_\omega)_N)$ that is $\Sigma_1(\alpha)$ definable from ${\cal P}(\aleph_\omega)_N$ and bounded on all sets in $L_\kappa({\cal P}(\aleph_\omega)_N)$. If $k:N\prec H(\aleph_\omega)$, as in Proposition \ref{cofinalfixed} we can extend it to $\hat{k}:{\cal P}(\aleph_\omega)_N\to {\cal P}(\aleph_\omega)$, and via $\rho_N$ we can even try to extend it to a $\bar{k}:L_\kappa({\cal P}(\aleph_\omega)_N)\to L_\kappa({\cal P}(\aleph_\omega))$: since for any $x\in L_\kappa({\cal P}(\aleph_\omega)_N)$ there is a $a\in{\cal P}(\aleph_\omega)_N$ such that $\rho_N(a)=x$, we can define $\bar{k}(x)=\rho(\hat{k}(x))$. If $k$ actually extends to an elementary embedding on $L_\kappa({\cal P}(\aleph_\omega)_N)$, as $\rho_n$ and $\rho$ are defined with the same formula using $\alpha$ as a parameter, $\bar{k}$ is the only possible extension such that $\bar{k}(\alpha)=\alpha$. Contrary to the case in Proposition \ref{cofinalfixed}, however, if $k$ it is not extendible it is not clear even if $\bar{k}$ is a function, as it can be that $\rho_N(a)=\rho_N(b)$ but $\rho(\hat{k}(a)\neq\rho(\hat{k}(b))$. 

Let ${\cal E}$ be the set of pairs $(N,k)\in L({\cal P}(\aleph_\omega))$ such that
\begin{itemize}
 \item $N\subseteq H(\aleph_\omega)$, $N=(H(\aleph_\omega))^{L(N)}$;
 \item $\kappa$ is the least such that $L_\kappa({\cal P}(\aleph_\omega)_N)\prec_1 L_{\Theta^{L({\cal P}(\aleph_\omega)_N)}}({\cal P}(\aleph_\omega)_N)$.
 \item $k:N\prec H(\aleph_\omega)$;
 \item $k$ induces $\bar{k}:L_\kappa({\cal P}(\aleph_\omega)_N)\to L_\kappa({\cal P}(\aleph_\omega))$;
 \item there exists $e:{\cal P}(\aleph_\omega)_N\twoheadrightarrow\alpha$, $e\in L_\kappa({\cal P}(\aleph_\omega)_N)$, such that $\bar{k}(\alpha)=\alpha$.
\end{itemize}

This is the point where we need the $\omega$-closure of $\mathbb{P}$: since $\mathbb{P}$ is $\omega$-closed, then 
\begin{multline*}
({\cal P}(\aleph_\omega)_{H(\aleph_\omega)})^{V[G]}=(\{\bigcup_{n\in\omega}a_n:(\forall n\ a_n\subseteq\aleph_n\wedge a_n\in H(\aleph_\omega))\wedge\forall n\leq m\ a_n\subseteq a_m\})^{V[G]}=\\
 =\{\bigcup_{n\in\omega}a_n:(\forall n\ a_n\subseteq\aleph_n\wedge a_n\in H(\aleph_\omega))\wedge\forall n\leq m\ a_n\subseteq a_m\}={\cal P}(\aleph_\omega). 
\end{multline*}
Therefore $(H(\aleph_\omega),j\upharpoonright H(\aleph_\omega))\in j({\cal E})$, thanks to Definition \ref{genericI0}(5), Lemma \ref{level} and by definition of $\kappa$, so ${\cal E}\neq\emptyset$.

If $\mathbb{P}$ were not $\omega$-closed, then it is conceivable that $({\cal P}(\aleph_\omega)_{H(\aleph_\omega)})^{V[G]}\neq {\cal P}(\aleph_\omega)$ and so $(H(\aleph_\omega),j\upharpoonright H(\aleph_\omega))\in j({\cal E})$ and ${\cal E}=\emptyset$.

For each $(N,k)\in{\cal E}$, let 
\begin{equation*}
 C_{(N,k)}=\{\eta<\kappa:\cof(\eta)=\omega,\ \bar{k}(\eta)=\eta\},
\end{equation*}
and for each $\sigma\subseteq{\cal E}$, $|\sigma|\leq\aleph_\omega$, let $C_\sigma=\bigcap_{(N,k)\in\sigma}C_{(N,k)}$. Let $F_0'$ the filter generated by the $C_\sigma$'s. 

\begin{claim}
 There exists $T\subseteq E^\kappa_\omega$ such that $F_0'\upharpoonright T$ is an ultrafilter
\end{claim}
\begin{proof}[Proof of Claim]
 As in Proposition \ref{measure}: Suppose that there exists a partition $\langle S_\xi:\xi<\crt(j)\rangle$ of $F_0$-positive sets. Then 
 \begin{equation*}
  j(\langle S_\xi:\xi<\crt(j)\rangle)=\langle T_\xi:\xi<j(crt(j))\rangle 
 \end{equation*}
is a partition of $E^\kappa_\omega$ in $j(F_0')$-positive sets. By definition, $C_{(H(\aleph_\omega),j\upharpoonright H(\aleph_\omega))}=\{\eta\in E^\kappa_\omega:j(\eta)=\eta)\}\in j(F_0')$, therefore $T_{\crt(j)}\cap C_{(H(\aleph_\omega),j\upharpoonright H(\aleph_\omega))}\neq\emptyset$. Let $\eta\in T_{\crt(j)}\cap C_{(H(\aleph_\omega),j\upharpoonright H(\aleph_\omega))}$. There must be some $\xi<\crt(j)$ such that $\eta\in S_\xi$. But then $\eta=j(\eta)\in T_\xi\cap T_{\crt(j)}$, contradiction. By the proof of the Remark \ref{stronglyornot}, using $\DC_{\aleph_\omega}$ and the $\aleph_\omega$-completeness of $F_0'$, the Claim is proved. 
\end{proof}

\begin{claim}
 \label{ultra}
 Let $F$ be the club filter on $E^\kappa_\omega$. Then $F_0'\upharpoonright T=F\upharpoonright T$.
\end{claim}
\begin{proof}
 It suffices to show that $F_0'\subseteq F$: in this case, $T$ is also stationary, so $F\upharpoonright T$ is a filter that extends $F_0'\upharpoonright T$, and since this is an ultrafilter they must be equal.
 
 Let $\sigma\subseteq{\cal E}$, $|\sigma|\leq\aleph_\omega$ then, we must prove that $C_{\sigma}$ is an $\omega$-club. For any $(N,k)\in{\cal E}$, by definition of ${\cal P}(\aleph_\omega)_N$ the function that associates $N$ to $L_\kappa({\cal P}(\aleph_\omega)_N)$ is $\Sigma_1(\alpha)$-definable from ${\cal P}(\aleph_\omega)$. Therefore the same is true for the function that associates $(N,k)$ to $\bar{k}$. Now let, for any $\beta<\kappa$, $G_\beta(N,k)=\bar{k}(\beta)$. Clearly $G_\beta$ is $\Sigma_1(\alpha)$-definable from ${\cal P}(\aleph_\omega)$, and therefore by Lemma \ref{boundedness} $G_\beta$ is bounded on $\sigma$, and therefore $\sup\{\bar{k}(\beta):(N,k)\in\sigma\}<\kappa$.
 
 \begin{subclaim}
  For every $(N,k)\in{\cal E}$, $L_\kappa({\cal P}(\aleph_\omega)_N)$ is closed under $\omega$-sequences.
 \end{subclaim}
 \begin{proof}
  It is a relativization to $L({\cal P}((\aleph_\omega)_N))$ of the fact that $\cof(\kappa)>\aleph_\omega$. If $\langle a_n:n\in\omega\rangle$ is such that $a_n\in L_\kappa({\cal P}(\aleph_\omega)_N)$ for any $n\in\omega$, then each $a_n$ is $\mathbf{\Delta}^2_1(\alpha)$ from ${\cal P}(\aleph_\omega)_N$. Each $a_n$ can be coded as a subset $A_n$ of ${\cal P}(\aleph_\omega)_N$, so that the $\omega$-sequence of the codes $A$ (that, in turn, can be coded as one subset $A$ of ${\cal P}(\aleph_\omega)_N$) codes $\langle a_n:n\in\omega\rangle$. Then $A$ is also $\mathbf{\Delta}^2_1(\alpha)$, and therefore is in $L_\kappa({\cal P}(\aleph_\omega)_N)$. By developing the code, one can find $\langle a_n:n\in\omega\rangle$ in $L_\kappa({\cal P}(\aleph_\omega)_N)$.
 \end{proof}

 Let $(N,k)\in\sigma$ and let $\beta_0<\kappa$. Define $\beta_{n+1}=\sup_{(N,k)\in\sigma}G_{\beta_n}(N,k)$ and $\beta_\omega=\sup_{n\in\omega}\beta_n$. Then $\bar{k}(\beta_\omega)=\sup_{n\in\omega}\bar{k}(\beta_n)=\beta_\omega$. Therefore $\beta_\omega\in C_\sigma$ and $C_\sigma$ is unbounded in $\kappa$.
 
 As $L_\kappa({\cal P}(\aleph_\omega)_N)$ is closed under $\omega$-sequences, it is trivial that $C_\sigma$ is $\omega$-closed.
\end{proof}

 Finally, we need to prove that $F_0'\upharpoonright T$ is $\kappa$-complete. Note that then also $F\upharpoonright T$ will be $\kappa$-complete, so the measurability of $\kappa$ will come from the club filter, in the spirit of $\omega$-strong measurability.
 
 Let $\langle A_\xi:\xi<\beta\rangle$ with $\beta<\kappa$ such that $A_\xi\subseteq T$ and $A_\xi\in F_0'$. Note that each element of ${\cal E}$ can be coded as an element of ${\cal P}(\aleph_\omega)$, and the same is true for $\aleph_\omega$-sequences of elements of ${\cal E}$. Fix therefore a surjection $G:{\cal P}(\aleph_\omega)\twoheadrightarrow{\cal E}^{\aleph_\omega}$, and for any $a\in{\cal P}(\aleph_\omega)$ write $C_{G(a)}$ simply as $C_a$.
 
 Fix a surjection $\pi:{\cal P}(\aleph_\omega)\twoheadrightarrow\beta$ (remember that $\kappa=\Theta^{L_\kappa({\cal P}(\aleph_\omega))}$), a bijection $\tau:{\cal P}(\aleph_\omega)\to{\cal P}(\aleph_\omega)\times {\cal P}(\aleph_\omega)$ and let $\pi'=\pi\circ\tau_1$. Let $\eta<\kappa$ witness the Coding Lemma in $L_\kappa({\cal P}(\aleph_\omega))$ for $\pi'$, as in Lemma \ref{codingink}. Let 
\begin{equation*}
 A=\{a\in{\cal P}(\aleph_\omega): C_{\tau_2(a)}\subseteq A_{\pi'(a))}\}. 
\end{equation*}
Since for all $\xi<\beta$ $A_\xi\in F_0'$, each $A_\xi$ contains some $C_b$, therefore $\pi'[A]=\beta$. Then there exists a $B\subseteq A$ that is in $L_\eta({\cal P}(\aleph_\omega))$ and such that $\pi`[B]=\beta$. 
 
 Let ${\cal E}_B=\bigcup\{G(\tau_2(a)):a\in B\}$, and for any $\delta<\kappa$ let 
\begin{equation*}
 W_\delta=\{\bar{k}(a):(N,k)\in{\cal E}_B,\ \rho_N(a)=\delta\}. 
\end{equation*}
Then
\begin{equation*}
 \begin{split}
  \rho''W_\delta & =\{\rho(\bar{k}(a)):(N,k)\in{\cal E}_B,\ \rho_N(a)=\delta\}\\
	               & =\{\bar{k}(\rho_N(a)):(N,k)\in {\cal E}_B,\ \rho_N(a)=\delta\}\\
								 & =\{\bar{k}(\delta):(N,k)\in{\cal E}_B\}. 
\end{split}
\end{equation*}
 As $W_\delta$ is in $L_\kappa({\cal P}(\aleph_\omega))$, by the boundedness of $\rho$ (\ref{boundednessofpi}) $\rho''W_\delta$ is bounded for any $\delta$. Now iterate the process as in \ref{ultra} to find a $\delta_\omega$ that is a fixed point for all the $k$'s such that $(N,k)\in{\cal E}_B$, and therefore $\delta_\omega\in C_{(N,k)}$. In other words, for any $a\in B$, $\delta_\omega\in C_{\tau_2(a)}$. 
 
 But for every $a\in B$, $C_{\tau_2(a)}\subseteq A_{\pi'(a)}$, and therefore $\delta_\omega\in A_{\pi'(a)}$. As $\pi'[B]=\beta$, $\delta_\omega$ is in all the $A_\xi$'s, and therefore the intersection of all the $A_\xi$'s is not empty, and therefore $F_0'\upharpoonright T$ is $\kappa$-complete.
\end{proof}

The added condition of $\omega$-closure for Theorem \ref{Thetalimit} seems to break the harmony of Theorem \ref{thm:main}. So we can ask:

\begin{open}
 Is it possible to prove Theorem \ref{Thetalimit} just from generic I0 at $\aleph_\omega$?
\end{open}

Part of the problem seems to be this: if $\mathbb{P}$ is $\omega$-closed, then ${\cal P}(\aleph_\omega)$ is a set that is easily describable fro $H(\aleph_\omega)$, in a certain sense $L({\cal P}(\aleph_\omega))^{V[G]}$ can compute it correctly. But if $\mathbb{P}$ is not $\omega$-closed, then we know that ${\cal P}(\aleph_\omega)$ is in $L({\cal P}(\aleph_\omega))^{V[G]}$ because of the hypotheses, but it is not possible to compute it there, it is just ``a'' set that extends $H(\aleph_\omega)$, so there is no way to characterize it. A similar problem arises in I0: if there exists $j:L(V_{\lambda+1})\prec L(V_{\lambda+1})$, then if $\mathbb{P}$ is a forcing notion that adds an $\omega$-sequence in $\lambda$ and $G$ is $\mathbb{P}$-generic, then $V_{\lambda+1}\notin L(V_{\lambda+1})^{V[G]}$. Therefore there is the real possibility that this Open Problem is connected to Open Problem \ref{perfect}.

\section{Open problems}

The most obvious question is of course:

\begin{open}
 Is generic I0 at $\aleph_\omega$ consistent under large cardinals?
\end{open}

But this is not the only direction that the research can follow. Consider in general $L({\cal P}(\lambda))$. One of the key points in the paper is that $L({\cal P}(\aleph_\omega))\nvDash\AC$. One can ask for which $\lambda$'s is true that $L({\cal P}(\lambda))\nvDash\AC$. Such question is non trivial only when $\lambda$ is singular, as it is easy to achieve when $\lambda$ is regular via forcing. But:

\begin{teo}[Shelah, \cite{Shelah2}]
 If $\lambda$ is singular with uncountable cofinality, then $L({\cal P}(\lambda))\vDash\AC$.
\end{teo}

Therefore it makes sense only for $\lambda$'s of cofinality $\omega$. Now, under \AD{} $L({\cal P}(\omega))\nvDash\AC$. If I0 holds at $\lambda$ then $L({\cal P}(\lambda))\nvDash\AC$. And if generic I0 holds at $\aleph_\omega$ then $L({\cal P}(\aleph_\omega))\nvDash\AC$. Are there other possibilities?

\begin{open}
 For which $\lambda$'s is true that $L({\cal P}(\lambda))\nvDash\AC$?
\end{open}

Another way to push forward the research on generic I0 would be to exploit the many results we already have on I0. For example, in \cite{Cramer} the inverse limits method is developed. In the generic I0 case is more difficult, as the original inverse limits method uses the fact that domain and codomain of the elementary embedding are the same, but it is possible to reflect the codomain inside the domain, as it is done in the proof of \ref{cofinalfixed}. 

\begin{open}
 Is there an inverse limit for generic I0?
\end{open}

Or one can investigate the degree structure of $L({\cal P}(\aleph_\omega))$, as it is done in \cite{Shi} for I0. 

\begin{open}
 How is $L({\cal P}(\aleph_\omega))$ under generic I0 at $\aleph_\omega$? Does Degree Determinacy hold?
\end{open}

The fact that in \ref{cofinalfixed} is proven that $j$ can come from an ideal opens the possibility of defining proper and non-proper elementary embeddings. Those were introduced in \cite{Woodin} and extensively studied in \cite{Dimonte} and \cite{Dimonte2} for hypotheses stronger than I0. Such hypotheses are of the form $j:L(N)\prec L(N)$, where $V_{\lambda+1}\prec N\prec V_{\lambda+2}$. Woodin in \cite{Woodin} defined a hierarchy of such $N$'s, called $E^0_\alpha$, so that the existence of an embedding $j:L(E^0_\alpha)\prec L(E^0_\alpha)$ is stronger the higher $\alpha$ is. For example, $E^0_0=L(V_{\lambda+1})\cap V_{\lambda+2}$ and $E^0_1=L(V_{\lambda+1},(V_{\lambda+1})^\sharp)\cap V_{\lambda+2}$, if it exists.

\begin{open}
 How a generic $E^0_\alpha$ could be defined? Does it exist a non-proper elementary embedding in such a setting?
\end{open}

Finally:

\begin{open}
 Is it possible to have generic I0 and $2^{\aleph_\omega}>\aleph_{\omega+1}$ in $V$?
\end{open}

(This was studied in \cite{DimonteFriedman}, \cite{DimonteWu} and \cite{ShiTrang} for I0).

\emph{Acknowledgments}. The paper was developed thanks to the generous financial support of the FWF (Austrian Science Fund) through projects M 1514-N25 and P 27815-N25, and also under the Italian program ``Rita Levi Montalcini 2013''. The author wants to thank the Kurt G\"{o}del Research Center, the Institute of Discrete Mathematics and Geometry at Technische Universit\"{a}t Wien and the Department of Mathematical, Computer and Physical Sciences at University of Udine for their hospitality, and the Department of Mathematics at Harvard University for the invitation that kickstarted this work.

This paper would not exist without the invaluable help of W. Hugh Woodin, whose input was the foundation of the entire project. The author would also like to thank Rachid Atmai, Scott Cramer, Sakae Fuchino, Mohammed Golshani, Daisuke Ikegami, Toshimichi Usuba, Liuzhen Wu for the many meaningful discussions that helped shape this work. Finally, we are grateful to the referee for their constructive comments and suggestions.

\end{document}